\documentclass[12pt]{amsart}
\usepackage{graphicx}
\usepackage{mathrsfs}
\usepackage{epsfig}
\usepackage{amsmath}
\usepackage{amsfonts}
\usepackage{amssymb}
\usepackage{amssymb,amscd}
\usepackage{tikz}
\usepackage{verbatim}
\usepackage{booktabs}
\usepackage[all]{xy}
\usepackage{multirow}

\newtheorem{thm}{Theorem}[section]

\newtheorem{prop}[thm]{Proposition}

\theoremstyle{definition}
\newtheorem{defn}[thm]{Definition}
\theoremstyle{remark}

\numberwithin{equation}{section}

\renewcommand{\leq}{\leqslant}
\renewcommand{\geq}{\geqslant}

\newcommand{\hc}{\mathbf{H}^2_{\mathbb{C}}}
\begin{document}
	
	\date{Oct. 4, 2023}

	\title[]{The topology of the Eisenstein-Picard modular surface}
	\author[J. Ma]{Jiming Ma}
	\address{School of Mathematical Sciences, Fudan University, Shanghai, China}
	\email{majiming@fudan.edu.cn}
	
	\author[B. Xie]{Baohua Xie}
	\address{School of Mathematics, Hunan University, Changsha, China}
	\email{xiexbh@hnu.edu.cn}
	
	
	
	\thanks{Jiming Ma was  supported by NSFC (No.12171092), he is  also  a member of LMNS, Fudan University. Baohua Xie was supported by NSFC (No.11871202).}
	
    \begin{abstract}The Eisenstein-Picard modular surface $M$ is the quotient space of the complex hyperbolic plane  by the  modular group $\rm PU(2,1; \mathbb{Z}[\omega])$. We determine the global topology of  $M$ as a 4-orbifold.

    \end{abstract}

     \maketitle
	
\section{Introduction}

  Complex hyperbolic $n$-space ${\bf H}^{n}_{\mathbb C}$ is the unique complete simply
connected K\"ahler  $n$-manifold with all holomorphic  sectional curvatures $-1$. 
The   holomorphic isometry group of ${\bf H}^{n}_{\mathbb C}$ is  $\rm PU(n,1)$. The Riemannian sectional curvatures of ${\bf H}^{n}_{\mathbb C}$  vary between  $-1$ and $-1/4$. It turns out that   complex hyperbolic manifolds  are rather difficult to study.

Due to the
extremal properties of ratios of their Chern numbers, complex hyperbolic manifolds are also of special interest in complex geometry.  
One of the most important questions in complex hyperbolic geometry is
the existence  of (infinitely many commensurable classes of)  non-arithmetic complex hyperbolic lattices, but we know very few examples of them  \cite{DeligneMostow:1986, Deraux:2020, dpp:2016, dpp:2021,Margulis,Fisher:2021} now.

  Even through it is well-known   there are infinitely many commensurable classes of arithmetic  complex hyperbolic manifolds in any complex dimension at least two, but the explicit  algebraic, geometric and topological information of them are still very mysterious. For example, Picard modular surfaces  ${\bf H}^2_{\mathbb C} /\rm PU(2,1; \mathcal{O}_d)$ are the simplest complex hyperbolic lattices. But   the geometries and algebras  of them are achieved until recently only  for very few   $d$    \cite{deraux-Xu, FalbelFrancsicsParker, falbelparker, MarkPaupert}. Here $\mathcal{O}_d$ is the ring of integers in the imaginary quadratic number field $\mathbb{Q}(\sqrt{-d})$. For the algebras of fake complex projective planes, a tiny portions of the second type  arithmetic complex hyperbolic lattices, see 	\cite{CartwrightSteger,PrasadYeung}.  We still need more explicit examples of complex hyperbolic lattices before  sketching a global picture of them (if we can eventually).

In this paper, we begin the study of the topologies of complex hyperbolic lattices. We consider a complex dimension one example before outline our result. 

\subsection{A toy example} 

The 2-disk $\mathbb{D}^2$ can be  identified  as $$\{z \in \mathbb{C}: |z| \leq 1 \}.$$
For each $n\in \mathbb{Z}_{+}$,  there is a smooth 	$\mathbb{Z}_{n}=\langle f \rangle$-action on  $\mathbb{D}^2$, namely $$f(z)={\rm e}^{\frac{2 \pi {\rm i}}{n}} \cdot z . 
$$ The quotient space of  $\mathbb{D}^2$ by this $\mathbb{Z}_{n}$-action is a 2-orbifold $\mathcal{F}_n$, there is a $\mathbb{Z}_n$-coned (or $2\pi/n$-coned) point $c_{n}$ in it.

Consider the well-known  modular surface $\Sigma$, where $\Sigma$ is the quotient space  of the complex hyperbolic line   ${\bf H}^1_{\mathbb C}$ by the modular group $\rm PSL(2; \mathbb{Z})$.  Recall that $\rm PSL(2; \mathbb{Z})$ is generated by $a$ and $b$, with $a^2=b^3=id$ and $ab$ is parabolic.
 So $\Sigma$ is a ${\bf H}^1_{\mathbb C}$-orbifold, with a $\mathbb{Z}_2$-coned point,   a  $\mathbb{Z}_3$-coned point, and a cusp. Take a simple closed  curve  $C$ in $\Sigma$   surrounding the cusp.  Then $\Sigma -C \times(0, \infty)$ is a compact orbifold, the interior of it is topologically homeomorphic to $\Sigma$, so it is a \emph{compact core} of $\Sigma$. We can understand the topology of $\Sigma$ by understanding the topology of $\Sigma -C \times(0, \infty)$.

$\Sigma -C \times(0, \infty)$ can be obtained with
\begin{itemize} 
	\item first we take two  2-orbifolds  $\mathcal{F}_2$ and $\mathcal{F}_3$. Which are small closed neighborhoods of the $\mathbb{Z}_2$-coned point and     $\mathbb{Z}_3$-coned point in $\Sigma$ respectively;

	\item then we glue a disk $[0,1]\times [0,1]$ to $\mathcal{F}_2 \cup \mathcal{F}_3$  to get a connected orientable 2-orbifold: $\{0\}\times [0,1]$ is identified with an arc in the boundary of $\mathcal{F}_2$, and $\{1\}\times [0,1]$ is identified with an arc in the boundary of $\mathcal{F}_3$.
	\end{itemize}
So we view   $\mathcal{F}_2$ and $\mathcal{F}_3$ as orbifold 0-handles, the  glued disk $[0,1]\times [0,1]$ as a 1-handle. Then we give   $\Sigma -C \times(0, \infty)$ a topological  handle structure.  See Figure \ref{figure:realmodular} for the process to get  a  handle structure of  $\Sigma -C \times(0, \infty)$. Where the left subfigure of  Figure \ref{figure:realmodular} is  the Ford domain of  $\rm PSL(2; \mathbb{Z})$ on  ${\bf H}^1_{\mathbb C}$ (in the upper half space model).  The boundary of the Ford domain  consists of arcs of Euclidean circles with radius 1 centered  at integers. 
The reader may find it is helpful to keep this example in mind.


\begin{figure}
	\begin{center}
		\begin{tikzpicture}
		\node at (0,0) {\includegraphics[width=11cm,height=3.2cm]{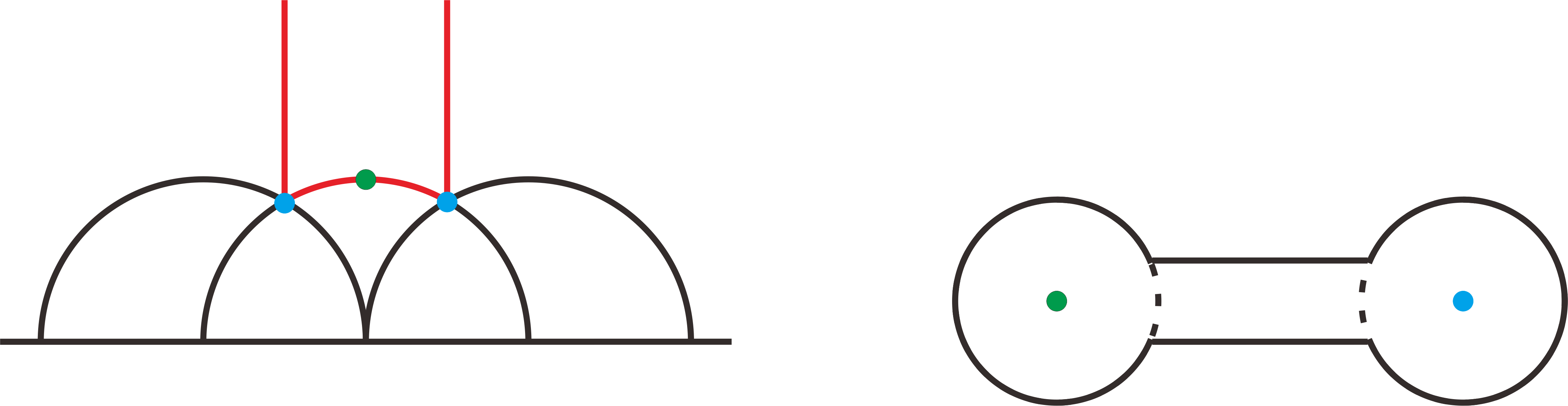}};
		
		\node at (1.9,-1.12){\tiny $\mathcal{F}_2$};
		\node at (4.8,-1.12){\tiny $\mathcal{F}_3$};

		\end{tikzpicture}
	\end{center}
	\caption{A handle decomposition of the   modular surface $\Sigma$. In the left sub-figure, the region bounded by the red arcs in a fundamental domain of $\rm PSL(2; \mathbb{Z})$'s action on  ${\bf H}^1_{\mathbb C}$. In the right sub-figure, we show the handle structure of a compact core of $\Sigma$. }
	\label{figure:realmodular}
\end{figure}

 In this paper, we will consider a similar problem for the more complicated ${\bf H}^2_{\mathbb C}$-orbifolds.

\subsection{Main result of the paper}

We determine the global topology of the  Eisenstein-Picard modular surface as a 4-orbifold  in this paper. 

Let $\rm PU(2,1; \mathbb{Z}[\omega]) \subset \rm PU(2,1)$ be the   Eisenstein-Picard modular group, where $\omega=(-1+ \rm{i} \sqrt{3})/2$ is a cube roots of unity. The group  $\rm PU(2,1; \mathbb{Z}[\omega])$ 
is a generalization of $\rm PSL(2; \mathbb{Z})$ into complex dimension two. We denote the Eisenstein-Picard modular surface by $M$,  which is the quotient space of the complex hyperbolic plane ${\bf H}^2_{\mathbb C}$ by the group $\rm PU(2,1; \mathbb{Z}[\omega])$. 

 Consider a \emph{compact core} $N$ of $M$. That is,  $N$ is a compact 4-orbifold with boundary a 3-orbifold, such that the interior of $N$ is homeomorphic to $M$.  Our main result is 
\begin{thm} \label{thm:modulartop} Let $M$ be the   Eisenstein-Picard modular surface, and $N$ a compact core of $M$.  Then $N$ can be obtained topologically  as:
\begin{enumerate}
\item [Step 1.]   Describe the topologies of four orbifold 0-handles $h(w_3)$, $h(w_4)$,  $h(w_{12})$ and $h(z_0)$. 

\item [Step 2.]  Attach  a 1-handle  $h([w_3,w_4])$ and  an    orbifold 1-handle  $h([z_0,w_{12}])$. 
\item [Step 3.] Attach an object $h([z_0,w_4,w_{12}])$,  the product of a  pair of   pants and the  2-orbifold $\mathcal{F}_2$, 
to get a 4-orbifold with 0-dimensional and  2-dimensional singular sets. 

\item [Step 4.]  Attach two  2-handles $$h([z_0,w_4,w_3,J(w_4)])\quad \text{and}  \quad \quad ~~~h([z_0,J(w_4), z_1,w_{12}]).$$


 \item [Step 5.] Attach a 3-handle  $$h([w_{12}, z_0, z_1, z_2, (w_3, w_4, J(w_4), PJ(w_4))]).$$

\end{enumerate}	

\end{thm}

Any one of the notations in Theorem  \ref{thm:modulartop},  such as $h(z_{i})$, $h(w_{j})$ and $h([z_0,J(w_4), z_1,w_{12}])$ etc.,   has a  geometrical meaning. Which are  related to cells of a fundamental domain $D^{*}$ of  $\rm PU(2,1; \mathbb{Z}[\omega])$, see  Subsection \ref{subsection:handleEPmodular}
 for the details.  For more on orbifold $k$-handles, see Subsection  \ref{subsection:handle}. 
 
 Handle theory is a standard tool for the study of topologies of  manifolds. By the well-known Selberg's Lemma, there is a finite cover $M'$  of $M$ such that  $M'$ is a 4-manifold (so without singular set).   But there is an order 72 torsion subgroup of  $\rm PU(2,1; \mathbb{Z}[\omega])$, so the covering degree $M^{'} \rightarrow M$ is at least 72.  In turn  the topology of $M'$ seems complicated (to the authors). So it seems  better to study the topology of $M$ instead of $M'$. Then we introduce  the terms orbifold $k$-handles for $k=0,1,2$, which are   mild (or even trivial) generalizations of $k$-handles. Moreover, we believe orbifold $k$-handles are  natural and convenient tools to study the topologies  of 4-orbifolds. We  hope this will not plague the reader.

The procedure  to get the topology of $N$ in   Theorem \ref{thm:modulartop} stage by stage  is sketched in Figure \ref{figure:EPmodular}. 

In Step 1, each of $w_3$, $w_4$, $w_{12}$ and $z_0$ is the fixed point of an isotropy group of order $3$, $4$, $12$ and $72$ respectively. The  associated  orbifold 0-handles  are just small closed neighborhoods of  them in $M$. See   Subsections  \ref{subsection:w3}, \ref{subsection:w4}, \ref{subsection:w12}, \ref{subsection:z0} for the topologies of these orbifold 0-handles.

In Step 2,    the  1-handle  $h([w_3,w_4])$ is simple, we just attach $\mathbb{D}^1 \times \mathbb{D}^3$ to  orbifold 0-handles $h(w_3)$, $h(w_4)$ along $\partial \mathbb{D}^1 \times \mathbb{D}^3$. The orbifold 1-handle  $h([z_0,w_{12}])$ is  a little tricky, which is  $\mathbb{D}^1 \times (\mathbb{D}^1 \times \mathcal{F}_6)$. So there is  $\mathbb{Z}_6$-coned singular set in it.  Moreover, the $\mathbb{Z}_6$-coned singular sets  of    $\partial \mathbb{D}^1 \times (\mathbb{D}^1 \times \mathcal{F}_6)$ are glued to  (subsets of) $\mathbb{Z}_6$-coned singular sets of the boundaries of $h(w_4)$ and $h(w_{12})$ respectively.

In Step 3, $h([z_0,w_4,w_{12}])$ is the product of a   pair of  pants and  the 2-orbifold $\mathcal{F}_2$, so it  is a $\mathbb{Z}_2$-coned 4-orbifold.  Three  copies of $\mathbb{S}^1 \times \mathcal{F}_2$  are gluing regions of   $h([z_0,w_4,w_{12}])$,  which are glued to $\mathbb{Z}_2$-coned singular sets of the boundaries of  $h(w_4)$, $h(w_{12})$ and  $h(z_{0})$ respectively. 


In Step 4,  the 2-handles  $h([z_0,w_4,w_3,J(w_4)])$ and $h([z_0,J(w_4), z_1,w_{12}])$  are attached  along two well-chosen  curves on the 4-orbifold obtained after Step 3. These information can be formulated  in the  fundamental domain $D^{*}$ (in Section  \ref{sec:EPmodularhandle} ) of  ${\bf H}^2_{\mathbb C}/\rm PU(2,1; \mathbb{Z}[\omega])$.

The 3-handle in Step 5 does not affect the fundamental group of the 4-orbifold after step 4. 


On page 2 of \cite{deraux},  for another Picard modular surface ${\bf H}^2_{\mathbb C}/\rm PU(2,1; \mathcal{O}_7)$, M. Deraux  wrote "This gives us some detailed information about the local structure of the quotient orbifold......, its global structure is still not understood ".    We believe Theorem \ref{thm:modulartop}  answers Derarux's question for ${\bf H}^2_{\mathbb C}/\rm PU(2,1; \mathcal{O}_3)$.  Moreover, 	there are some papers which relating some  ${\bf H}^2_{\mathbb C}$-lattices to other (in some sense simple) complex surfaces (and curve arrangements in   complex surfaces) \cite{deraux2018, deraux2019, Koziarz-RR}.  To our knowledge,  Theorem \ref{thm:modulartop}  is one of the very few results (if not the first one) about explicit topology structures of ${\bf H}^2_{\mathbb C}$-lattices.

\begin{figure}
	\begin{center}
		\begin{tikzpicture}
		\node at (0,0) {\includegraphics[width=8cm,height=11cm]{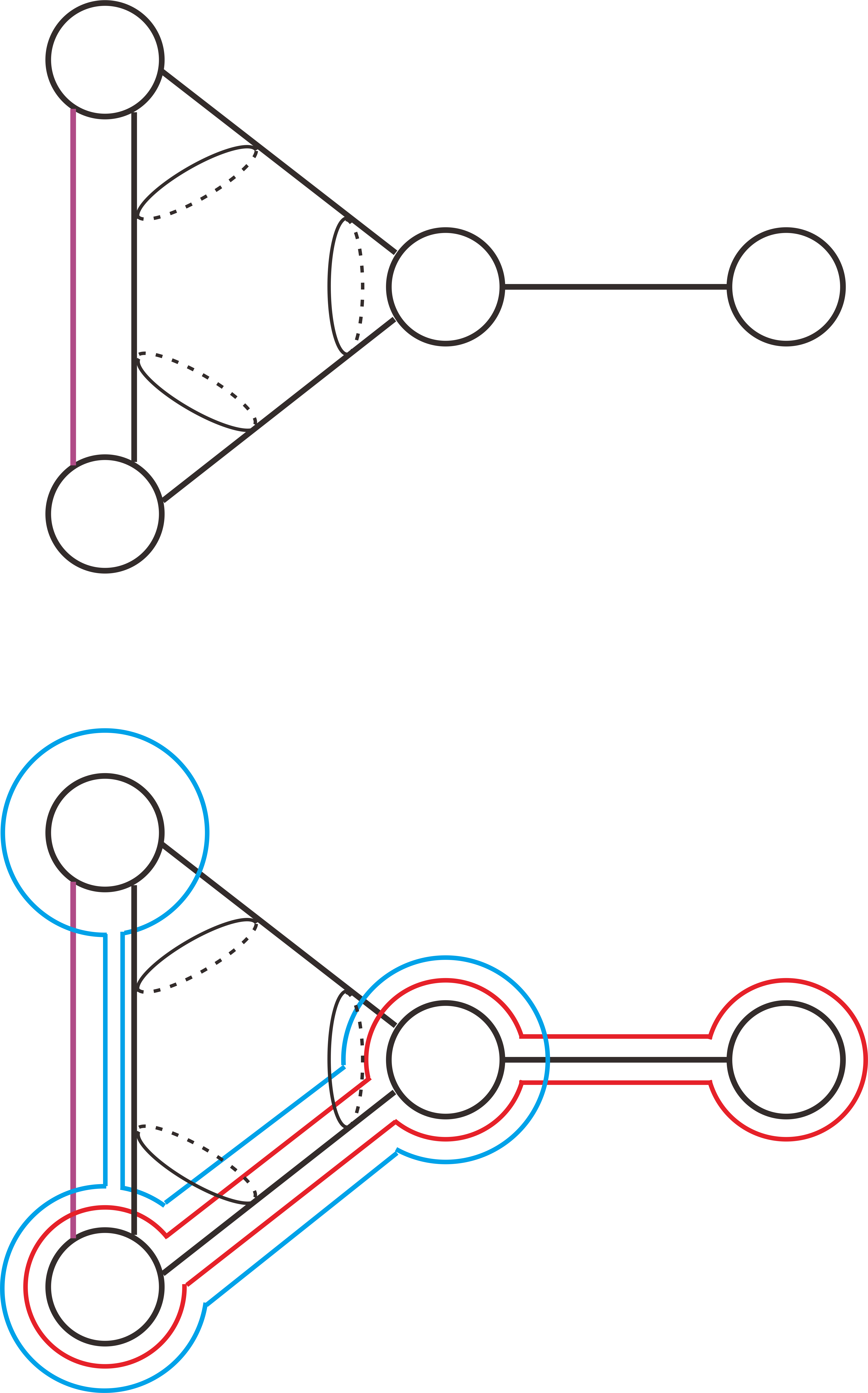}};
		
		\node at (0.1,3.2){\large $w_4$};
		\node at (0.1,2.2){\tiny $\langle R_1R_2R_3 \rangle $};

		\node at (3.2,3.2){\large $w_3$};
			\node at (3.2,4.2){\tiny $\langle RP \rangle $};
		
			\node at (-3.1,4.95){\large $w_{12}$};
			\node at (-1.3,4.95){\tiny $\langle PQ^{-1}, R\rangle$};
				\node at (-3.1,1.3){\large $z_0$};
				\node at (-1.1,0.7){\Tiny $\langle R_1, R_2J^2, R_2R_3R^{-1}_2 \rangle $};
				
				\node at (-1.85,3.3){\tiny $[z_0,w_4,w_{12}]$};
				
				\node at (-4.01,3.3){\Tiny $[z_0,w_{12}]$};

			\node at (0.1,-2.9){\large $w_4$};
				\node at (3.2,-2.9){\large $w_3$};
					\node at (-3.1,-1.1){\large $w_{12}$};
						\node at (-3.1,-4.7){\large $z_0$};
		\end{tikzpicture}
	\end{center}
	\caption{A handle decomposition of the  Eisenstein-Picard modular surface. The upper sub-figure is the 4-orbifold after Step 3 of  Theorem \ref{thm:modulartop}. There is a pair of $\mathbb{Z}_2$-coned totally geodesic pants $[z_0,w_4,w_{12}]$ in it. We also list the isotropic groups for $z_0$, $w_3$, $w_4$ and $w_{12}$ in the  upper sub-figure.  The $[z_0,w_{12}]$ labeled arc indicates the $\mathbb{Z}_6$-orbifold 1-handle.  In the bottom sub-figure, we attach two 2-handles along the blue and the red curves. There is no information of the  3-handle in this figure since a 3-handle  is really simple.}
	\label{figure:EPmodular}
\end{figure}
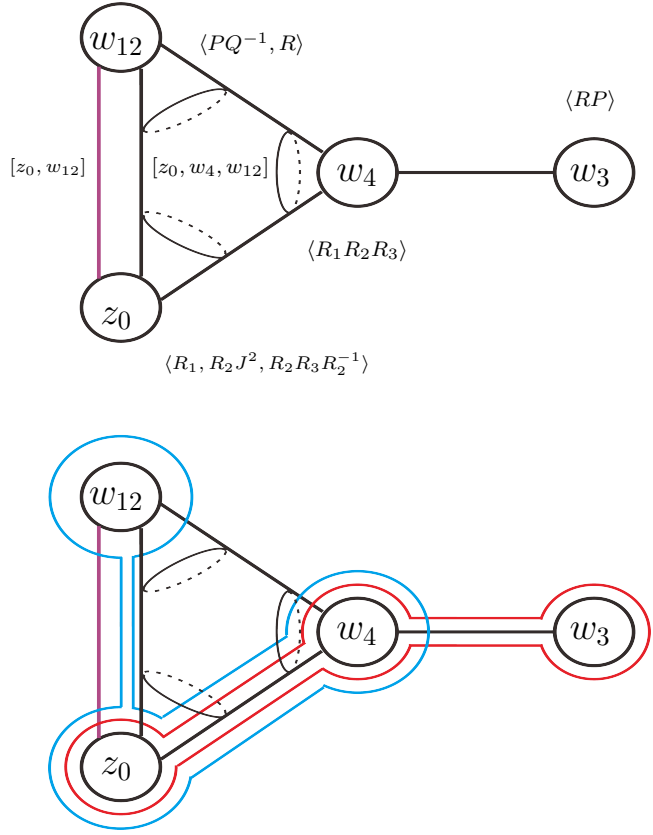


Our proof of Theorem \ref{thm:modulartop} depends essentially on the fundamental domain $D$ of   $\rm PU(2,1; \mathbb{Z}[\omega])$'s action on ${\bf H}^2_{\mathbb C}$  by Falbel-Parker \cite{falbelparker}.  The fundamental domain $D$ by Falbel-Parker, combinatorically a 4-simplex, is extremely simple. But at the same time, it misleads the complexity  of the  Eisenstein-Picard modular surface. There are three 
isolated points of isotropy groups, namely $w_3$, $w_4$ and $w_{12}$,  hiding in the edges and ridges of the 4-simplex $D$.  So we subdivide $D$ into a new fundamental domain $D^{*}$. With the help of $D^{*}$, we can get Theorem \ref{thm:modulartop}.

 {\bf The paper is organized as follows.} In Section \ref{sec:background}  we give well known background
material on complex hyperbolic geometry and some information about the fundamental domain $D$ constructed in \cite{falbelparker}. In Section \ref{sec:localmodel}, we consider local topologies of isolated points of isotropy groups in $\rm PU(2,1; \mathbb{Z}[\omega])$.
In Section \ref{sec:EPmodularhandle}  we give the handle structure of the  Eisenstein-Picard modular surface based on the cells of the new fundamental domain $D^{*}$, so  we prove Theorem \ref{thm:modulartop}.

	\textbf{Acknowledgement}: Part of the work was carried out when Jiming Ma was  visiting  Hunan University in the summer of  2023, the hospitality  is gratefully appreciated. The second named author is grateful to LMNS(Laboratory
	of Mathematics for Nonlinear Science) of Fudan University for its hospitality and financial support.
	
\section{Preliminary}  \label{sec:background}

The purpose of this section is to introduce briefly complex hyperbolic geometry. One can refer to Goldman's book \cite{Go} for more details.  We also  sketch  some properties of   $\rm PU(2,1; \mathbb{Z}[\omega])$  after Falbel-Parker \cite{falbelparker}.

\subsection{Complex hyperbolic plane}

Let $\langle {\bf{z}}, {\bf{w}} \rangle={\bf{w}^{*}}H{\bf{z}}$ be the Hermitian form on ${\mathbb{C}}^3$ associated to $H$, where $H$ is the Hermitian matrix
$$
H=\left[
\begin{array}{ccc}
0 & 0 & 1 \\
0 & 1 & 0 \\
1 & 0 & 0 \\
\end{array}
\right]
$$
of signature $(2,1)$.
Then ${\mathbb{C}}^3$ is the union of the negative cone $V_{-}$, null cone $V_{0}$ and positive cone $V_{+}$, where
\begin{eqnarray*}
	V_{-} &=& \left\{ {\bf{z}}\in {\mathbb{C}}^3\backslash\{0\} : \langle {\bf{z}}, {\bf{z}} \rangle <0 \right\}, \\
	V_{0} &=& \left\{ {\bf{z}}\in {\mathbb{C}}^3\backslash\{0\} : \langle {\bf{z}}, {\bf{z}} \rangle =0 \right\}, \\
	V_{+} &=& \left\{ {\bf{z}}\in {\mathbb{C}}^3\backslash\{0\} : \langle {\bf{z}}, {\bf{z}} \rangle >0 \right\}.
\end{eqnarray*}

\begin{defn}
	Let $P: {\mathbb{C}}^3\backslash\{0\} \rightarrow {\mathbb{C}P^2}$ be the projectivization map.
	Then the \emph{complex hyperbolic plane} $\hc$ is defined to be $P(V_{-})$,  and its {boundary} $\partial \hc$ is defined to be $P(V_{0})$.
	The \emph{Bergman metric} on the complex hyperbolic plane is given by the distance formula
	\begin{equation}  \label{eq:bergman-metric}
	\cosh^2\left(\frac{d(u,v)}{2}\right)=\frac{\langle {\bf{u}}, {\bf{v}} \rangle\langle {\bf{v}}, {\bf{u}} \rangle}{\langle {\bf{u}}, {\bf{u}} \rangle \langle {\bf{v}}, {\bf{v}} \rangle},
	\end{equation}
	where ${\bf{u}}, {\bf{v}} \in {\mathbb{C}}^3$ are lifts of $u,v$. Here $d(u,v)$ is the distance between two points $u,v \in \hc$.
\end{defn}


The standard lift  \begin{equation}\nonumber
\left[
\begin{array}{c}
z_1 \\
z_2 \\
1 \\
\end{array}
\right] 
\end{equation}
of $z=(z_1,z_2)\in \mathbb {C}^{2}$ is negative if and only if
$$z_1+|z_2|^2+\overline{z}_1=2{\rm Re}(z_1)+|z_2|^2<0.$$  Thus $\mathbb{P}(V_{-})$ is a paraboloid in ${\mathbb C}^{2}$, called the {\it Siegel domain}.
In these coordinates, the boundary $\mathbb{P}(V_{0})$ is given by
$$2{\rm Re}(z_1)+|z_2|^2=0.$$

Let $\mathcal{N}=\mathbb{C}\times \mathbb{R}$ be the \emph{Heisenberg group} with product
$$
[z,t]\cdot [\zeta,\nu]=[z+\zeta,t+\nu-2{\rm{Im}}(\bar{z}\zeta)].
$$
We write $q=[z,t]\in\mathcal{N}$ for $z \in \mathbb{C}$ and $t \in \mathbb{R}$. Then the boundary of the complex hyperbolic plane $\partial \hc$ can be identified with the union $\mathcal{N}\cup \{q_{\infty}\}$, where $q_{\infty}$ is the point at infinity.
The \emph{standard lift} of $q_{\infty}$ and $q=[z,t]\in\mathcal{N}$ in $\mathbb{C}^3$ are
\begin{equation}\label{eq:lift}
{\bf{q}_{\infty}}=\left[
\begin{array}{c}
1 \\
0 \\
0 \\
\end{array}
\right] \quad \text{and} \quad
{\bf{q}}=\left[
\begin{array}{c}
\frac{-|z|^2+it}{2} \\
z \\
1 \\
\end{array}
\right].
\end{equation}


Complex hyperbolic plane and its boundary $\hc \cup \partial \hc$ can be identified with ${\mathcal{N}}\times{\mathbb{R}_{\geq 0}}\cup q_{\infty}$.
Any point $q=(z,t,u)\in{\mathcal{N}}\times{\mathbb{R}_{\geq0}}$ has the standard lift
$$
{\bf{q}}=\left[
\begin{array}{c}
\frac{-|z|^2-u+it}{2} \\
z \\
1 \\
\end{array}
\right].
$$
Here $(z,t,u)$ is called the \emph{horospherical coordinates} of $\overline {\bf H}^2_{\mathbb{C}}=\hc \cup \partial \hc$. The natural projection $\mathcal{N}=\mathbb{C} \times \mathbb{R} \rightarrow \mathbb{C}$ is called the {\it vertical projection}.

For $v >0$, the \emph{horoball} $H_{>v}$ of height $v$ is the open subspace of ${\bf H}^2_{\mathbb{C}}$ in horospherical coordinates  $(z,t,u)$ with $u > v$. The boundary of $H_{>v}$, that is,   the subspace of ${\bf H}^2_{\mathbb{C}}$ in horospherical coordinates  $(z,t,v)$,  is the \emph{horosphere} of height $v$.

\subsection{The isometries }The complex hyperbolic plane is a K\"{a}hler manifold of constant holomorphic sectional curvature $-1$.
We denote by $\rm U(2,1)$ the Lie group of $\langle \cdot,\cdot\rangle$ preserving complex linear
transformations and by $\rm PU(2,1)$ the group modulo scalar matrices. The group of holomorphic
isometries of ${\bf H}^2_{\mathbb C}$ is exactly $\rm PU(2,1)$. It is sometimes convenient to work with
$\rm SU(2,1)$, which is a 3-fold cover of $\rm PU(2,1)$.

Elements of $\rm SU(2,1)$ fall into three types, according to the number and types of the fixed points of the corresponding
isometry. Namely, an isometry is {\it loxodromic} (resp. {\it parabolic}) if it has exactly two fixed points (resp. one fixed point)
on $\partial {\bf H}^2_{\mathbb C}$. It is called {\it elliptic}  when it has (at least) one fixed point inside ${\bf H}^2_{\mathbb C}$.
An elliptic $A\in \rm SU(2,1)$ is called {\it regular elliptic} whenever it has three distinct eigenvalues, and {\it special elliptic} if
it has a repeated eigenvalue.

Any unipotent element of $\rm SU(2,1)$ is conjugate in $\rm SU(2,1)$
to one fixing $q_{\infty}$ given by:
$$T_{[z,t]}=\left(\begin{matrix}
1 & -\overline{z}& \frac{-|z|^{2}+it}{2} \\ 0 & 1 & z \\ 0 & 0 & 1 \end{matrix}\right).$$
Note that applying  $T_{[z,t]}$ to $[w,s]$ amounts to doing the Heisenberg
left multiplication by $[z,t]$. For that reason $T_{[z,t]}$ is called a 
{\it Heisenberg translation}.  A Heisenberg translation by $[0,t]$ is called a {\it vertical translation} by $t$.

The full stabilizer of $q_{\infty}$ is generated by the above unipotent group, together with the isometries of the forms
\begin{equation}
\left(\begin{matrix}
1 & 0& 0 \\ 0 & e^{i\theta} & 0 \\ 0 & 0 & 1 \end{matrix}\right) \quad  and   \quad
\left(\begin{matrix}
\lambda & 0 & 0 \\ 0 & 1 & 0 \\ 0 & 0 & 1/\lambda \end{matrix}\right),
\end{equation}
where $\theta,\lambda\in \mathbb{R}$ and $\lambda \neq 0$.  The first acts on $\partial {\bf H}^2_{\mathbb C}\backslash\{q_{\infty}\}=\mathbb{C}\times \mathbb{R}$ as a rotation
with vertical axis:
$$(z,t)\mapsto (e^{i\theta}z,t),$$  whereas the second one acts as $$(z,t)\mapsto (\lambda z,\lambda^2 t).$$
Note that the parabolic isometries fixing $q_{\infty}$ are Cygan isometries, see \cite{Go}.

\subsection{Totally geodesic submanifolds and complex reflections}
There are two kinds of totally geodesic submanifolds of real dimension 2 in ${\bf H}^2_{\mathbb C}$: {\it complex lines} in ${\bf H}^2_{\mathbb C}$ are
complex geodesics (represented by ${\bf H}^1_{\mathbb C}\subset {\bf H}^2_{\mathbb C}$) and {\it Lagrangian planes} in ${\bf H}^2_{\mathbb C}$ are totally
real geodesic 2-planes (represented by ${\bf H}^2_{\mathbb R}\subset {\bf H}^2_{\mathbb C}$). Since the Riemannian sectional curvature of the  complex hyperbolic plane is nonconstant, there are no totally geodesic hypersurfaces.  A {\it polar vector} of a complex line $L$  is the unique vector (up to scalar multiplication) perpendicular to this complex line with
respect to the Hermitian form. A polar vector  belongs to  $V_{+}$ and each vector in $V_{+}$ corresponds to a complex line.



There is a special class of elliptic elements in $\rm PU(2,1)$.
\begin{defn}
	The \emph{complex reflection} of order $k$ about the  complex line $C$ with polar vector ${\bf{n}}$ is given by the following formula:
	\begin{equation}\label{eq:involution}
	I_{C}({\bf{z}})=-{\bf{z}}+(1- {\rm e ^{\frac{2 \pi {\rm i}}{k}}})\frac{2\langle {\bf{z}}, {\bf{n}} \rangle}{\langle {\bf{n}}, {\bf{n}} \rangle} {\bf{n}}.
	\end{equation}
	Then  $I_{C}$ is a holomorphic isometry fixing the complex line $C$ pointwisely.
\end{defn}

\subsection{Isometric spheres,  spinal coordinates and Ford polyhedron}


\begin{defn}
	The \emph{(extended) Cygan metric} on $\hc$ is given by 
	\begin{equation}\label{eq:cygan-metric-extend}
	d_{\textrm{Cyg}}(p,q)=\left| |z-w|^2+|u-v|-i(t-s+2{\rm{Im}}(z\bar{w})) \right|^{1/2},
	\end{equation}
	where $p=(z,t,u)$ and $q=(w,s,v)$ are in horospherical coordinates.
\end{defn}

Suppose that $g=(g_{ij})^3_{i,j=1}\in \rm PU(2,1)$ does not fix $q_{\infty}$. This implies that $g_{31}\neq 0$, see Lemma 4.1 of \cite{Par}.
\begin{defn}
	The \emph{isometric sphere} of $g$, denoted by $ I(g)$, is the set
	\begin{equation}\label{eq:isom-sphere}
	I(g)=\{ p \in {\hc \cup \partial\hc} : |\langle {\bf{p}}, {\bf{q}}_{\infty} \rangle | = |\langle {\bf{p}}, g^{-1}({\bf{q}}_{\infty}) \rangle| \}.
	\end{equation}
\end{defn}

The \emph{interior} of $I(g)$ is the set
\begin{equation}\label{eq:exterior}
\{ p \in {\hc \cup \partial\hc} : |\langle {\bf{p}}, {\bf{q}}_{\infty} \rangle | > |\langle {\bf{p}}, g^{-1}({\bf{q}}_{\infty}) \rangle| \}.
\end{equation}
The \emph{exterior } of $I(g)$ is the set
\begin{equation}\label{eq:interior}
\{ p \in {\hc \cup \partial\hc} : |\langle {\bf{p}}, {\bf{q}}_{\infty} \rangle | < |\langle {\bf{p}}, g^{-1}({\bf{q}}_{\infty}) \rangle| \}.
\end{equation}

The isometric sphere $I(g)$ is the Cygan sphere with center
$$g^{-1}({\bf{q}}_{\infty})=\left[\frac{\overline{g_{32}}}{\overline{g_{31}}},2{\rm{Im}}(\frac{\overline{g_{33}}}{\overline{g_{31}}})\right]$$
and radius $r_g=\sqrt{\frac{2}{|g_{31}|}}$.


\begin{defn}The \emph{Ford domain} $D_{\Gamma}$ for a discrete group $\Gamma < \mathbf{PU}(2,1)$ centered at $q_{\infty}$ is
	the intersection of the (closures of the) exteriors of all isometric spheres of elements in $\Gamma$ not fixing $q_{\infty}$. That is, $D_{\Gamma}$ is the set
	$$\{p\in {\bf H}^2_{\mathbb C} \cup \partial{\bf H}^2_{\mathbb C}: |\langle \mathbf{p},q_{\infty}\rangle|\leq|\langle \mathbf{p},G^{-1}(q_{\infty})\rangle|
	\ \forall G\in \Gamma \ \mbox{with} \ G(q_{\infty})\neq q_{\infty} \}.$$
\end{defn}



The Ford  domain  is preserved by $\Gamma_{\infty}$, the stabilizer of
$q_{\infty}$ in $\Gamma$.  In this case, $D_{\Gamma}$ is only a fundamental domain modulo the action of $\Gamma_{\infty}$.  In other words,  the intersection of the Ford domain with a fundamental domain for $\Gamma_{\infty}$ is a fundamental domain for
$\Gamma$. Facets of co-dimension one, two, three and four in $D_{\Gamma}$ will be called {\it sides}, {\it ridges}, {\it edges} and {\it vertices}, respectively.

Once there is a guessed  Ford domain for a group $\Gamma$, then one  can check  it is in fact the Ford domain using the Poincar\'e polyhedron theorem. Moreover, the side-pairings and  the Poincar\'{e} polyhedron
theorem   give an abstract presentation of $\Gamma$, see \cite{dpp:2016}.


\subsection{A fundamental domain of  the   Eisenstein-Picard modular group $\rm PU(2,1; \mathbb{Z}[\omega])$} \label{subsection:FPdomain}

We outline the  fundamental domain of the   Eisenstein-Picard modular group  $\rm PU(2,1; \mathbb{Z}[\omega])$  by Falbel-Parker \cite{falbelparker}.

Now let $\Gamma=\rm PU(2,1; \mathbb{Z}[\omega])$, and $\Gamma_{\infty}$ the stabilizer of $q_{\infty}$ in $\Gamma$.

\begin{prop} \label{prop:gammainfty}  \cite{falbelparker} $\Gamma_{\infty}$ is generated by 
	\begin{equation}
	P=\left[\begin{matrix}	1 & 1& \omega \\ 0 & \omega & -\omega \\ 0 & 0 & 1 \end{matrix}\right] \quad  and  \quad  Q=
	\left[\begin{matrix}
	1 & 1 & \omega \\ 0 & -1 & 1 \\ 0 & 0 & 1 \end{matrix}\right].
	\end{equation}
	Moreover,  $\Gamma_{\infty}$  has the presentation 
	$$\Gamma_{\infty}=\left\langle P, Q \big| \begin{array}  {c}  (QP^{-1})^6=P^3Q^{-2}=id
	\end{array}\right\rangle.
	$$

\end{prop}

We also need a complex reflection of order two, say  
	\begin{equation}
R=\left[\begin{matrix}	0 & 0& 1\\ 0 & -1 & 0 \\ 1 & 0 & 0 \end{matrix}\right]
\end{equation}
in $\Gamma$.
We define $$J=RP$$ as an order-3 regular elliptic element. We also define  $$R_1=QP^{-1},\quad ~~~~R_2=JR_1J^{-1}=RPQP^{-2}R,\quad  ~~ R_3=J^{-1}R_1J=P^{-1}Q$$ as three $\mathbb{C}$-reflections of order-6 (note that there is a typo on page 285 of \cite{falbelparker}, where it stated $R_2=RPQ^{-1}P^{-2}R$). We  also note that $R=(JR^{-1}_1J)^2=(R_3R_1R_2)^2$.

For readers' convenience, we 
also list the polar vector of the mirror of $R_i$ for $i=1,2,3$. They are 
\begin{equation} \nonumber
n_1=\left[\begin{matrix}		0 \\ 1  \\ 0 \end{matrix}\right], \quad  \quad  n_2=\left[\begin{matrix}	0\\-\omega \\ 1 \end{matrix}\right] \quad \text {and} \quad   n_3=\left[\begin{matrix}	-1\\1\\ 0 \end{matrix}\right]
\end{equation}
respectively. Then the   polar vectors of the mirrors of $R_3R_1R^{-1}_3$,  $R_1R_2R^{-1}_1$ and $R_2R_3R^{-1}_2$ are 
\begin{equation} \nonumber
 \left[\begin{matrix}	\overline{\omega}\\1 \\ 0 \end{matrix}\right], \quad   \quad  \left[\begin{matrix}		0 \\ 1  \\ 1 \end{matrix}\right] \quad \text {and}  \quad \left[\begin{matrix}	-\omega\\0 \\ 1 \end{matrix}\right]
\end{equation}
respectively. 

Let
\begin{equation} \nonumber
z_0=\left[\begin{matrix}		\overline{\omega} \\ 0  \\ 1 \end{matrix}\right], \quad  \quad  z_1=\left[\begin{matrix}	-1\\-\omega  \\ 1 \end{matrix}\right], \quad  \quad  z_2=\left[\begin{matrix}	-1\\1\\ 1 \end{matrix}\right]  \quad  and  \quad  z_3=\left[\begin{matrix}	\omega \\0 \\ 1 \end{matrix}\right]
\end{equation}
be four points in $\mathbf{H}^2_{\mathbb C}$. 
 Then checking directly, we have 
 \begin{prop} \label{prop:reflection} The reflections have the following properties:
\begin{itemize}
	
	\item $R_1$ fixes $z_0$, $z_3$ and $q_{\infty}$. Moreover, $R_1(z_1)=z_2$;
	\item $R_2$ fixes $z_1$;
	\item $R_3$ fixes  $z_2$ and $q_{\infty}$;
		\item 
	 $R_2R_3R_2^{-1}$  fixes $z_0$;
	 	\item 
	 $R_3R_1R_3^{-1}$  fixes  $z_1$ and $q_{\infty}$;
	 
	 	\item 
	 $R_1R_2R_1^{-1}$  fixes  $z_2$;
	
	\item 
 $R_2$ and $R_3R_1R_3^{-1}$ commute;
 
 	\item 
 $R(z_0)=z_3$,  $R(z_3)=z_0$, $R(z_1)=z_1$ and $R(z_2)=z_2$.
\end{itemize}
\end{prop} 

 Falbel-Parker  \cite{falbelparker}
constructed a fundamental domain  $D$ of $\Gamma$ on $\mathbf{H}^2_{\mathbb C}$, which is a 4-simplex with vertices $q_{\infty}$ and $z_i$ for $i=0,1,2,3$. So in other words, $D$ can be viewed as a geometrical realization of a 4-simplex in  $\overline{\mathbf{H}^2_{\mathbb C}}$. We will denote by  $[z_0,q_{\infty}]$ the edge in $D$ spanned by $z_0$ and $q_{\infty}$, and by  $[z_0, z_3, z_{1}]$  the ridge in $D$ spanned by  $z_0$, $z_3$ and  $z_{1}$, and  etc. 
We will not need the detailed geometric information of $D$ except for the following:

\begin{prop} \label{prop:geoofD}  \cite{falbelparker} For  the geometric polytope   $D$  in $\mathbf{H}^2_{\mathbb C}$:
	\begin{itemize}

	\item    $z_i$ for $i=0,1,2,3$ span a tetrahedron $T_{0}$ in the isometric sphere of $RP$  in $\mathbf{H}^2_{\mathbb C}$; $D$ is the geodesic cone  from $q_{\infty}$ to $T_{0}$;
		\item Each edge  of $D$ is a geodesic;
		 
		\item The ridge $[z_0,z_3, q_{\infty}]$ lies in a  $\mathbb{C}$-line. 
			\item  The ridges $[z_0,z_3, z_{1}]$ and   $[z_0,z_3, z_{2}]$ lie  in   $\mathbb{R}$-planes. Moreover, $R_1([z_0,z_3, z_{1}])=[z_0,z_3, z_{2}]$.
		
	   \end{itemize}
\end{prop}

Using the fundamental domain $D$, Falbel-Parker proved
\begin{thm} \label{thm:grouppresentation}  \cite{falbelparker}    $\Gamma$  has the following equivalent presentations:
	\begin{enumerate}
		
		\item    $$\left\langle P, Q, R \big| \begin{array}  {c} R^2= (QP^{-1})^6= PQ^{-1}RQP^{-1}R=P^3Q^{-2}=(RP)^3=id
		\end{array}\right\rangle;
		$$
		\item  	$$\left\langle R_1, R_2, R_3 \Bigg| \begin{array}  {c} R_{k}^6=id, R_{k+1}R_{k}R_{k+1}=R_{k}R_{k+1} R_{k}, ~~k \in \{1,2,3\}\\ [3 pt]
		(R_1R_2R_3)^{4}=id, (R_1R_2R_3)^{-2} R_1R_2 =(R_2R_3R_1)^{-2} R_2R_3
		\end{array}\right\rangle;
		$$
		\item 	
		$$\left\langle J, R_1 \big| \begin{array}  {c} J^3=R_1^6= (JR^{-1}_1J)^4= R_1(JR^{-1}_1J)^2R^{-1}_1(JR^{-1}_1J)^{-2}=id
		\end{array}\right\rangle.
		$$
		
	\end{enumerate}
\end{thm}

\section{Local topologies of isolated fixed points of isotropy groups} \label{sec:localmodel}
The   conjugacy classes of torsion elements in some
Picard modular groups were well understood in \cite{deraux-Xu}. In this section, we consider the topologies around the  isolated fixed points of isotropy groups.

Recall if a  group $G$  acts on a  space $X$, and   $x \in X$ be a point, then $$G_{x}= \{ g \in G: gx=x\}$$ is  called the \emph{isotropy group} associated to $x$. 
For a point $z \in {\bf H}^2_{\mathbb C}$, we denote by $[z]$ its projection to $M={\bf H}^2_{\mathbb C}/\Gamma$. We will consider the topology  of a closed small neighborhood of $[z]$ in $M={\bf H}^2_{\mathbb C}/\Gamma$
when $\Gamma_{z}$ is a nontrivial subgroup.

\subsection{Local topology around $[w_{3}]$} \label{subsection:w3}

Recall $J=RP$ in Subsection \ref{subsection:FPdomain}. We denote by $w_{3}$ the fixed point of $J$. It can be checked directly that  \begin{equation} \nonumber
 w_3=\left[\begin{matrix}	-{\rm e}^{\frac{\pi \rm i}{9}} \\  \frac{1}{2}-(\frac{\sqrt{3}}{2}-2 \sin(\frac{\pi}{9})) \cdot  {\rm i}  \\ 1 \end{matrix}\right].
 \end{equation}

 Moreover we have  $J(z_0)=z_1$, $J(z_1)=z_2$, $J(z_2)=z_0$.  
  Using the notations in  \cite{falbelparker}, where the 2-face of $D$  with vertices $z_0, z_1$ and $z_2$  is denoted by  $F_3$. 
   With the geographical coordinates of $F_3$ on page 270 of \cite{falbelparker}, it can be showed $w_{3}$ lies  in  the interior  of the  2-face $F_3$.  See Figure 	\ref{figure:fundamentaldomain}. 
 Let $\Gamma_{w_{3}}$ be the stabilizer of $w_{3}$ in $\Gamma$. From  \cite{falbelparker} and \cite{deraux-Xu}, $\Gamma_{w_{3}}$ is generated by $J$  with order 3. We take a small closed  neighborhood  $N([w_{3}])$ of  $[w_{3}]$   in $M$.  Then $N([w_{3}])$  is a  4-orbifold with non-empty boundary. The boundary of  $N([w_{3}])$ is also called the \emph{vertex link} of $N([w_{3}])$ in $M$, which  is  denoted by $Y_{3}$.  $Y_{3}$ is a spherical 3-orbifold, and $N([w_{3}])$  is a cone on $Y_3$. So $Y_3$ completely determines the topology of $N([w_{3}])$.

\begin{prop} \label{prop:w12} $Y_{3}$ is the lens space $L(3,-1)$.
\end{prop}
\begin{proof} $Y_{3}$ is a  3-orbifold with fundamental group $\mathbb{Z}_3$. Since $J$ is a regular elliptic element, then $Y_{3}$ is a 3-manifold. That is, there is no singular set in $Y_{3}$. So $Y_{3}$ is the lens space $L(3,-1)$.
	
\end{proof}

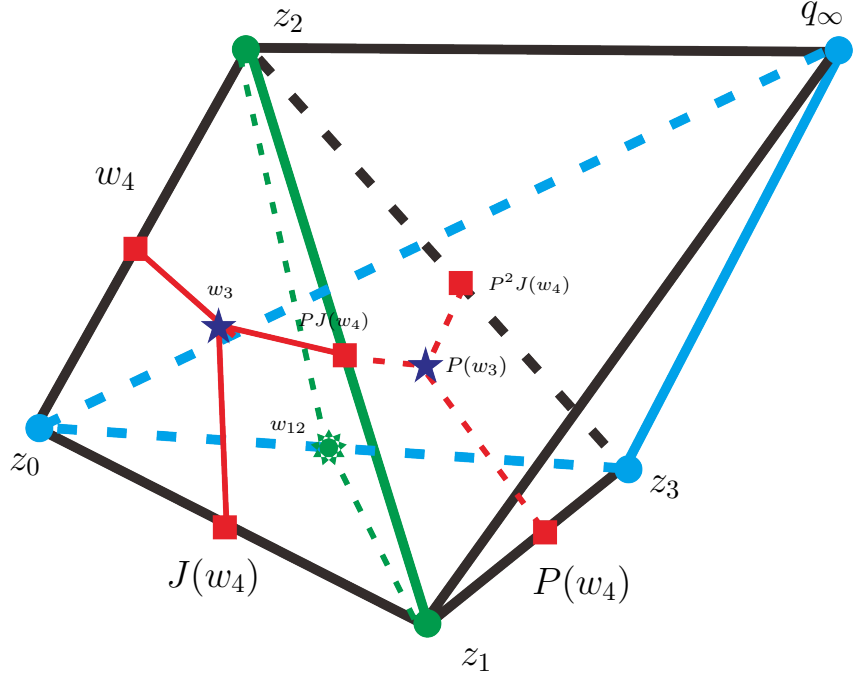
\begin{figure}
	\begin{center}
		\begin{tikzpicture}
		\node at (0,0) {\includegraphics[width=11cm,height=8cm]{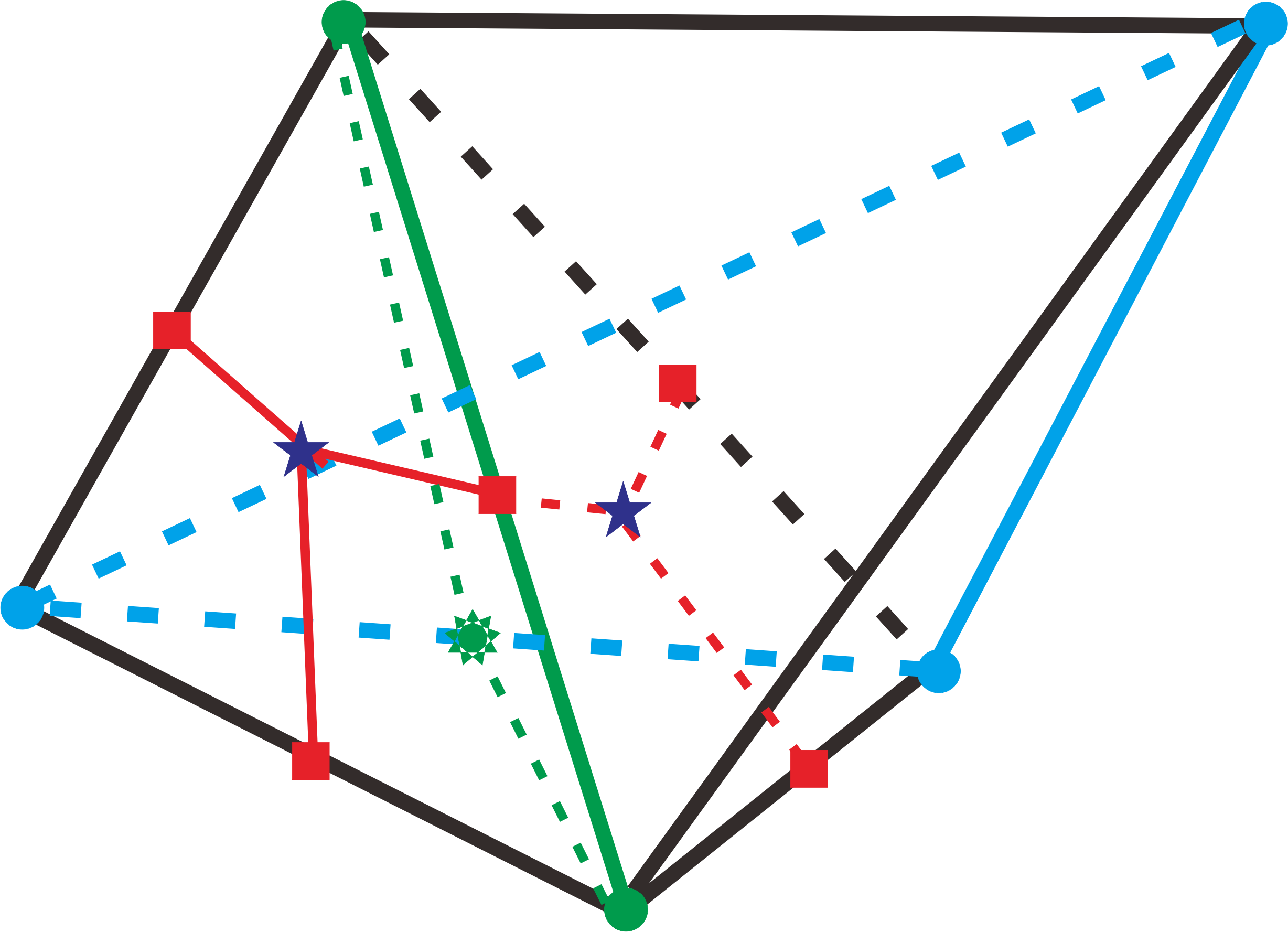}};
		
		\node at (5.1,4.3){\large $q_{\infty}$};
		
		\node at (-5.5,-1.7){\large $z_0$};
		\node at (0.5,-4.3){\large $z_1$};
		\node at (-2.0,4.2){\large $z_2$};
		\node at (3.0,-2.0){\large $z_3$};

		\node at (-2.9,0.6){\tiny $w_3$};
			\node at (0.5,-0.4){\tiny $P(w_3)$};
		\node at (-4.3,2.1){\large $w_4$};
		\node at (-3.0,-3.2){\large $J(w_4)$};
		\node at (-1.4,0.2){\Tiny $PJ(w_4)$};
		\node at (1.2,0.7){\Tiny $P^2J(w_4)$};
		\node at (1.9,-3.3){\large $P(w_4)$};
		\node at (-2.0,-1.2){\tiny $w_{12}$};

		\end{tikzpicture}
	\end{center}
	\caption{A fundamental domain $D$  of $\Gamma$,  with subdivisions of some of its faces to get new  fundamental domain $D^*$. Where each $f(w_3)$  is  a star-shaped vertex, and each  $f(w_4)$  is a  square-shaped vertex for certain $f$ in   $\rm PU(2,1; \mathbb{Z}[\omega])$.}
	\label{figure:fundamentaldomain}
\end{figure}

\subsection{Local topology around $[w_{4}]$}  \label{subsection:w4}


We denote by $w_{4}$ the fixed point of $R_1R_2R_3$.
 It can be checked directly that  \begin{equation} \nonumber
w_4=\left[\begin{matrix}	-\frac{\sqrt{3}}{2}-\frac{ {\rm i}}{2} \\  \frac{1}{2}+\frac{ (2-\sqrt{3}){\rm i}}{2}  \\ 1 \end{matrix}\right].
\end{equation}

 With the geographical coordinates of the geodesic arc $\gamma_{02}$ on page 268 of \cite{falbelparker}, it can be showed $w_{4}$ lies in the geodesic from $z_0$
to $z_2$.  Let $\Gamma_{w_{4}}$ be the stabilizer  of $w_{4}$ in $\Gamma$. From  \cite{falbelparker} and \cite{deraux-Xu}, $\Gamma_{w_{4}}$ is generated by $R_1R_2R_3$  with order 4. 

 We take a small closed neighborhood  $N([w_{4}])$ of  $[w_{4}]$   in $M$. The vertex link of $N([w_{4}])$ in $M$ is  $Y_{4}$, which is a spherical 3-orbifold.  $N([w_{4}])$  is a cone on $Y_4$. So $Y_4$ completely determines the topology of $N([w_{4}])$.

\begin{prop} \label{prop:w12} $Y_{4}$ is the lens space $L(4,-1)$.
\end{prop}
\begin{proof} $Y_{4}$ is a  3-orbifold with fundamental group $\mathbb{Z}_4$. Since $R_1R_2R_3$ is regular elliptic, then $Y_{4}$ is a 3-manifold. That is, there is no  singular set in $Y_{4}$ . So $Y_{4}$ is the lens space $L(4,-1)$.
\end{proof}


\subsection{Local topology around $[w_{12}]$}  \label{subsection:w12}

We denote by $w_{12}$ the intersection of the reflection mirrors of the complex reflections  $PQ^{-1}$ and $R$.  It can be checked directly that
 \begin{equation} \nonumber
w_{12}=\left[\begin{matrix}	-1 \\ 0  \\ 1 \end{matrix}\right],
\end{equation}
 which lies in the interior of the  geodesic from $z_0$
to $z_3$.  Let $\Gamma_{w_{12}}$ be the stabilizer of $w_{12}$ in $\Gamma$. From  \cite{falbelparker} and \cite{deraux-Xu}, $\Gamma_{w_{12}}$ is generated by $PQ^{-1}$ and $R$  with order 12. Moreover, $PQ^{-1}$ and $R$ commute.

We take a small closed  neighborhood  $N([w_{12}])$ of  $[w_{12}]$   in $M$.  The vertex link of $N([w_{12}])$ in $M$ is  $Y_{12}$, which is a spherical 3-orbifold  with   $2\pi/2$-coned and   $2\pi/6$-coned singular sets.  $N([w_{12}])$  is a cone on $Y_{12}$. So $Y_{12}$ completely determines the topology of $N([w_{12}])$.

\begin{prop} \label{prop:w12} $Y_{12}$ is a 3-orbifold with underlying space the 3-sphere. The singular set is the Hopf link, one component is   $2\pi/2$-coned and the other one is  $2\pi/6$-coned.
\end{prop}
\begin{proof} $Y_{12}$ is a  3-orbifold with fundamental group $\mathbb{Z}_2 \oplus \mathbb{Z}_6$. It is  the  3-orbifold in Table 6 on page 89 of \cite{dunbar} with $f=2$ and $g=6$. Moreover, it can be showed directly this orbifold has  fundamental group $\mathbb{Z}_2 \oplus \mathbb{Z}_6$.
\end{proof}

\subsection{Local topology around $[z_{0}]$}  \label{subsection:z0}


Recall  \begin{equation} \nonumber
z_{0}=\left[\begin{matrix}	\frac{-1-{\rm i} \sqrt{3} }{2} \\ 0  \\ 1 \end{matrix}\right],
\end{equation}  which is the intersection of the reflection mirrors of $R_1$ and $R_2R_3R^{-1}_2$.
Let $\Gamma_{z_0}$ be the stabilizer of $z_0$ in $\Gamma$. From  \cite{falbelparker} and \cite{deraux-Xu}, $\Gamma_{z_0}$ is generated by $R_1, R_2J^{2}$ and $R_2R_3R^{-1}_2$ with order 72.  If we take $$a=R_1,~~b=R_2R_3R^{-1}_2,~~~ c=R_2J^2.$$ Then as an abstract group, $\Gamma_{z_0}$ is 
$$\left\langle a, b,c \big| \begin{array}  {c} a^6=b^6=c^{12}=id, \quad ab=ba, \quad cac^{-1}
=b,a^5b^5=c^{2}\end{array}\right\rangle. $$

We take a small closed  neighborhood  $N([z_{0}])$ of  $[z_{0}]$   in $M$.   The vertex link of $N([z_0])$ in $M$ is  $Y_{0}$, which is a spherical 3-orbifold  with   $2\pi/2$-coned and   $2\pi/6$-coned singular sets.
$N([z_{0}])$  is a cone on $Y_{0}$. So $Y_{0}$ completely determines the topology of $N([w_{0}])$.

\begin{prop} \label{prop:z0} $Y_{0}$ is a 3-orbifold with underlying space the lens space $L(3,2)$. The singular set is the union of two knots in $L(3,2)$, one component is   $2\pi/2$-coned and the other one is  $2\pi/6$-coned.
\end{prop}
\begin{proof}Via Magma, we have $c^2$  generates the center of  $\Gamma_{z_0}$. Moreover, we have the exact sequence: $$1\rightarrow \mathbb{Z}_6 \rightarrow \Gamma_{z_0} \rightarrow \pi^{orb}(S^2(2,2,6))\rightarrow 1.$$	
Where $S^2(2,2,6)$ is the 2-orbifold on the 2-sphere with three coned points, and they are  $\mathbb{Z}_2$-coned, $\mathbb{Z}_2$-coned and $\mathbb{Z}_6$-coned respectively. 
	So $Y_{0}$ is a  Seifert spherical 3-orbifold  with  base 2-orbifold $S^2(2,2,6)$. Then it appeared in Table 4 in Page 837 of  \cite{Mecchia-Seppi2015}. Moreover, on page 1309 of \cite{Mecchia-Seppi2019}, it explained the relations between   Seifert invariants and the index of the singularity of a Seifert 3-orbifold. We then can check that $\Gamma_{z_0}$ is the group $(C_{12}/C_6, D^*_{24}/D^*_{12})$ in Family 4 of  Table 4 on page 837 of  \cite{Mecchia-Seppi2015} with $m=n=3$. The Seifert invariants of $Y_0$ are $(\frac{6}{6},\frac{3}{2},\frac{4}{2})$. Where $\frac{6}{6}$ and $\frac{4}{2}$ correspond to $2\pi/6$-coned and   $2\pi/2$-coned singular loci of $Y_0$ respectively (due to the fact $gcd(6,6)=6$ and $gcd(4,2)=2$).  The  underlying space of $Y_{0}$ is a  Seifert spherical 3-orbifold  based on the  2-orbifold $S^2(2)$  with invariant $\frac{3}{2}$, so it is the lens space $L(3,2)$.

	Moreover, we take a group  $$G=\left\langle c_1,c_2,c_3,h \big| \begin{array}  {c} h^6=c_1c_2c_3, \quad h^6=id, \quad c^2_1=h^3, \quad c^2_2=c^6_3=id
	\end{array}\right\rangle. $$
	Which is the orbifold fundamental group of a 3-orbifold $Y$ based on  2-orbifold $S^2(2,2,6)$ and Seifert invariants  $(\frac{6}{6},\frac{3}{2},\frac{4}{2})$. If we add the relation  $c_2=c_3=id$ in $G$, we get the group $\mathbb{Z}_3$, which is the fundamental group of  $L(3,2)$.  Moreover, it can be checked directly the group 
	$G$ is isomorphic to $\Gamma_{z_0}$.
	
	In total $Y_{0}$ is the 3-orbifold $Y$ as above.

	\end{proof}
\subsection{Local topology around $[q_{\infty}]$} 

We take a horosphere $H_{u}$ in  ${\bf H}^{2}_{\mathbb C}$ with  $u$ large enough,  $H_{> u}$ is the associated horoball. Then $\Gamma$ acts on  
$${\bf H}^{2}_{\mathbb C}- \mathop{\cup}_{\gamma \in \Gamma} \gamma \cdot  H_{ > u}$$
proper discontinuously, the quotient space is denoted by $M_{\leq u}$. We can  view $M_{\leq u}$ as the space as truncating the cusp of $M$. So $M_{\leq u}$ is a compact 4-orbifold with non-empty boundary,  and $M$ is homeomorphic to the interior of $M_{\leq u}$.

The boundary of $M_{\leq u}$, that is, the quotient space of $H_{u}$ by $\Gamma_{\infty}$,   is a  3-orbifold $Y_{\infty}$. The fundamental group of  $Y_{\infty}$ is  $\Gamma_{\infty}$. $H_{\geq u} /\Gamma_{\infty}$ is homeomorphic to $Y_{\infty} \times [0, \infty)$, so it can be viewed as the cone on  $Y_{\infty}$, but the cone point $q_{\infty}$ is deleted.

\begin{prop} \label{prop:zinfty} $Y_{\infty}$ is a 3-orbifold with underlying space the 3-sphere, and the singular set is the $2\pi/6$-coned trefoil  knot.
\end{prop}

\begin{proof} It is known $\Gamma_{\infty}$ is an extension of $\mathbb{Z}_6$ by $\pi^{orb}(S^2(2,3,6))$  \cite{falbelparker}. Where $S^2(2,3,6)$ is the 2-orbifold on the 2-sphere with three coned points, and they are  $\mathbb{Z}_2$-coned, $\mathbb{Z}_3$-coned and $\mathbb{Z}_6$-coned respectively. So $Y_{\infty}$ is a Seifert 3-manifold with   base $2$-orbifold $S^2(2,3,6)$. Moreover,  $Y_{\infty}$ is a 3-dimensional Nil-orbifold with exactly one component of singular set, it is a   $2\pi/6$-coned circle.
Dunbar 	 \cite{dunbar} classified all the 3-dimensional Nil-orbifold. There is exact one of them satisfies above conditions, that is, the left lower corner one in Table 2 on page 81 of 	 \cite{dunbar}.

In fact, 
if we take an orbifold $Y$ with underlying space the 3-sphere, and the singular set of $Y$ is the $2\pi/6$-coned trefoil knot. Then by \cite{Rolfsen},  the orbifold fundamental group  is $$\pi^{ord}(Y)=\left\langle x,y \big| \begin{array}  {c} x^6=id, ~~ xyx=yxy
\end{array}\right\rangle. $$
Take $$f: \Gamma_{\infty} \longrightarrow \pi^{ord}(Y),$$
with $f(P)=y^{-1}x^{-1}$ and $f(Q)=y^{-1}x^{-2}$. Then $f$ is an isomorphism  between $\Gamma_{\infty}$ and $\pi^{ord}(Y)$,   with $f^{-1}(x)=PQ^{-1}$ and $f^{-1}(y)=Q^{-1}P$.

So we proved in two ways $Y_{\infty}$ is the 3-orbifold $Y$ above. 
\end{proof}

\section{Topology  of Eisenstein-Picard modular surface $M$} \label{sec:EPmodularhandle}

In this section, after some preparations on orbifold handles, we subdivide Falbel-Parker's fundamental domain $D$ into a new one $D^{*}$. Based on side-pairings on $D^{*}$,   we prove Theorem \ref{thm:modulartop} in the last.

\subsection{An introduction to handle theory}\label{subsection:handle}

We give a sketch of handle theory following \cite{GompfStipsicz}. From now on, we assume all 4-orbifolds and  4-manifolds are orientable.

\begin{defn}\label{defn:handle} For $0 \leq k  \leq 4$, a 4-dimensional \emph{$k$-handle} $h$ is a copy of
	$\mathbb{D}^{k} \times \mathbb{D}^{4-k}$, attached to the boundary of an orientable smooth $4$-manifold $Y$ along $\partial{\mathbb{D}^{k}} \times \mathbb{D}^{4-k}$
	by an embedding $\phi:\partial\mathbb{D}^{k} \times \mathbb{D}^{4-k} \rightarrow \partial Y$, to get  an orientable smooth $4$-manifold $Y\cup_{\phi}h$.
\end{defn}

For a   4-dimensional $k$-handle $h$:
\begin{itemize}
	\item 	 $\mathbb{D}^{k} \times  \{0\}$ is called the \emph{core};
	\item  $ \{0\} \times \mathbb{D}^{4-k}$ is the \emph{cocore};
	
	\item $\phi$ is the \emph{attaching map};
	\item $\partial\mathbb{D}^{k} \times  \{0\}$ (or its image) is the \emph{attaching sphere};
	
	\item  $ \{0\} \times \partial\mathbb{D}^{n-k}$  is the \emph{belt
		sphere};
	\item  The number $k$ is called the \emph{index} of
	the handle.
\end{itemize}
Since  there is a deformation retraction of $Y\cup_{\phi}h$ onto $Y\cup_{\phi | _{\mathbb{D}^{k} \times  \{0\}}}\mathbb{D}^{k}\times  \{0\}$, we have that up to homotopy, attaching a $k$-handle is the same as attaching a $k$-cell.


A  4-dimensional $0$-handle is attached to $\partial Y$  along $\partial \mathbb{D}^{0} \times \mathbb{D}^{4}=\emptyset$,
so attaching a $0$-handle to a $4$-manifold is the same as taking the disjoint
union of $Y$ with  $\mathbb{D}^{4}$.

The attaching sphere of a $1$-handle is $\partial\mathbb{D}^{1} \times \{0\}$, which is a pair of points.
If $\partial Y$ is connected and nonempty, then there is a unique isotopy class of $\phi:\partial\mathbb{D}^{0} \times  \{0\} \rightarrow \partial X$ since we assume all the manifolds are orientable. If $\partial Y$ is disconnected, a  $1$-handle
attaching to two component of $\partial Y$ is just the boundary sum along two components of $\partial Y$.

Since $\pi_{2}(O(1))$ is the trivial group, there is a unique way to attach a 3-handle $\mathbb{D}^{3} \times \mathbb{D}^{1}$ along $\partial \mathbb{D}^{3} \times \mathbb{D}^{1}$ to  the boundary of a $4$-manifold $Y$.

A 4-handle $h$  is the same as  $\mathbb{D}^{4}$ attached along  $\partial\mathbb{D}^{4}=\mathbb{S}^{3}$.

As above,  4-dimensional $k$-handles are easy to understand, for $k=0, 1, 3$ and $4$. However, 2-handles are more complicated. For a 2-handle  $\mathbb{D}^{2} \times \mathbb{D}^{2}$, it is determined by an embedded  circle in   $\partial Y$ and an integer (so called ``framing"), they together corresponds to a Dehn surgery from $\partial Y$ to $\partial (Y\cup_{\phi}\mathbb{D}^{2} \times \mathbb{D}^{2})$, see Part II of \cite{GompfStipsicz}.

We give a  mild/trivial generalization of $k$-handles into orbifold $k$-handles.

First,  let  $L$ be a   closed orientable  3-orbifold with finite orbifold-fundamental group $G$. If  the singular set $\mathcal{S}$ of $L$ is non-empty, we  assume the singular set $\mathcal{S}$ is a disjoint union of circles (this is the case we will use in this paper).  So $L$ has  $\mathbb{S}^{3}$ as the universal cover, $L$ is a  \emph{$\mathbb{S}^{3}$-orbifold}. The cone on $L$  is a 4-orbifold $Y$, which can also be obtained from $\mathbb{D}^{4}$ by $G$-action with exactly one global  fixed point. Moreover, $Y$ has 2-dimensional singular set if and only if the singular set $\mathcal{S}$ of $L$ is non-empty.

\begin{defn}\label{defn:orbifold0handle}
	The cone $Y$ based on a  $\mathbb{S}^{3}$-orbifold $L$ is a 4-dimensional \emph{orbifold $0$-handle}. 
\end{defn}


Recall the 2-orbifold $\mathcal{F}_n$  with the cone point $c_{n}$ in the introduction, which is just the quotient space of $\mathbb{D}^2$ by $\mathbb{Z}_n$.  So the product $\mathbb{D}^1 \times \mathcal{F}_n$ is a 3-orbifold, with underlying space a 3-ball and  there is a  $\mathbb{Z}_{n}$-coned proper arc it. Then  $\mathbb{D}^1 \times (\mathbb{D}^1 \times \mathcal{F}_n)$  can be viewed as a 4-orbifold with 2-dimensional singular set $\mathbb{D}^1 \times \mathbb{D}^1 \times \{c_{n}\}$.


\begin{defn}\label{defn:orbifold1handle} A 4-dimensional \emph{orbifold $1$-handle} $h$ is a copy of  $\mathbb{D}^1 \times (\mathbb{D}^1 \times \mathcal{F}_n)$, attached to the boundary of an orientable smooth $4$-orbifold $Y$ along $\partial \mathbb{D}^1 \times (\mathbb{D}^1 \times \mathcal{F}_n)$
	by an embedding $\phi:\partial \mathbb{D}^1 \times (\mathbb{D}^1 \times \mathcal{F}_n) \rightarrow \partial Y$, such that  $\partial \mathbb{D}^1 \times (\mathbb{D}^1 \times \mathcal{F}_n)$ are identified with two copies of $\mathbb{D}^1 \times \mathcal{F}_n$ in the boundary of $Y$,  to get  an orientable smooth $4$-orbifold $Y\cup_{\phi}h$.

\end{defn}

\begin{defn}\label{defn:orbifold2handle}
	 A 4-dimensional \emph{orbifold  $2$-handle} $h$ is a copy of  $\mathbb{D}^2 \times  \mathcal{F}_n$, attached to the boundary of an orientable smooth $4$-orbifold $Y$ along $\partial \mathbb{D}^2 \times \mathcal{F}_n$
	by an embedding $\phi:\partial \mathbb{D}^2 \times \mathcal{F}_n \rightarrow \partial Y$, such that  $\partial \mathbb{D}^2 \times \mathcal{F}_n$ is identified with a copy of $\mathbb{S}^1 \times \mathcal{F}_n$ in the boundary of $Y$,  to get  an orientable smooth $4$-orbifold $Y\cup_{\phi}h$.
 
\end{defn}

\subsection{Equivalent classes of faces of the fundamental polytope $D$}\label{subsection:equaclassD}

Since  the fundamental polytope $D$ constructed by Falbel-Parker is a 4-simplex, we will denote a $k$-face of $D$ by its $(k+1)$ vertices. For example, $[z_0, z_1, z_3, q_{\infty}]$ is a 3-face of $D$. 
It is not difficult to get the  equivalent classes of faces of $D$ under  the action of  $\Gamma=\rm PU(2,1; \mathbb{Z}[\omega])$.

\begin{enumerate}
	\item There are two classes of vertices of $D$  under  the action of  $\Gamma$.  $D$ has five vertices $z_0$,  $z_1$,  $z_2$,  $z_3$ and $q_{\infty}$.
	Since $z_n=P^{n}(z_0)$ for $n=1,2,3$, $z_0$,  $z_1$,  $z_2$ and  $z_3$  are equivalent.  $q_{\infty}$ is an ideal vertex, which  is not    equivalent to any other $z_i$.

		\item There are three classes of edges of $D$  under  the action of $\Gamma$. $D$ has six finite edges $[z_i,  z_j]$ for $0 \leq i< j \leq 3$.  Five of them 
		 are equivalent, say any two of $[z_i,  z_j]$ for  $\{i,j\} \neq \{0,3\}$ are equivalent.
		 $[z_0,  z_3]$ is not   equivalent to any other edges.
		 $D$ has  four infinite edges $[z_i,  q_{\infty}]$ for $0 \leq i \leq 3$, any two of them are equivalent.

			\item There are four classes of ridges of $D$  under  the action of $\Gamma$. The ridges $[z_0, z_1,  z_2]$ and  $[z_3, z_1,  z_2]$ are in the same class. The ridges $[z_0, z_1,  z_3]$ and  $[z_0, z_2,  z_3]$ are in the same class. 
			 Any two $[z_{i}, z_{j},   q_{\infty}]$ are in the same class  for  $\{i,j\} \neq \{0,3\}$.
			  The ridge $[z_0, z_3,  q_{\infty}]$ is the only one in its class.

				\item   There are three classes of 3-faces of $D$  under  the action of  $\Gamma$. The 3-faces $[z_0, z_1,  z_2, q_{\infty}]$ and  $[z_1, z_2,  z_3, q_{\infty}]$ are in the same class. The 3-faces
				$[z_0, z_2,  z_3, q_{\infty}]$ and  $[z_0, z_1,  z_3, q_{\infty}]$ are in the same class. The 3-face
				$[z_0, z_1,z_2, z_3]$ is the only one in its class. 
				
	\end{enumerate}

\subsection{Subdivide  the fundamental polytope $D$  get a new  polytope $D^*$.}

From the fundamental polytope $D$ of $\Gamma$ constructed by Falbel-Parker, we add more vertices, edges and ridges, to subdivide some of its facets to get a new  polytope $D^*$. The polytope $D^*$ is crucial in our proof of Theorem \ref{thm:modulartop}.

We add eight new vertices in  the interiors of some edges or ridges of $D$ for the new  polytope $D^*$:
\begin{enumerate}
	
	\item We add a new vertex $w_4$ on  the edge $[z_0,z_2]$. Now the  edge $[z_0,z_2]$ is divided into two edges $[z_0,w_4]$ and $[z_2,w_4]$.

	 Since the edges   $[z_0,z_1]$, $[z_1,z_2]$, $[z_2,z_3]$, $[z_1,z_3]$ are in the same class of $[z_0,z_2]$, each of these edges contains a vertex  which is equivalent to $w_4$. They are $$J(w_4),~~PJ(w_4),~~P^2J(w_4), ~~P(w_4)$$  respectively.	In turn, we subdivide each of $[z_0,z_1]$, $[z_1,z_2]$, $[z_2,z_3]$, $[z_1,z_3]$ into two edges by one of these vertices. 
	 
	 
	 	By the side-pairing maps and the map $J$,   we can find that each sub-edge in the edges $[z_0,z_1]$, $[z_1,z_2]$, $[z_2,z_3]$, $[z_1,z_3]$ is equivalent to $[z_0,w_4]$ or $[z_2,w_4]$.

	\item We add a  vertex $w_{3}$ in the interior of the triangular ridge $[z_0,z_2,z_1]$.

	Since the ridge $[z_0,z_1,z_2]$ is paired to the ridge $[z_1,z_2,z_3]$ by $P$,  the ridge $[z_1,z_2,z_3]$ contains the corresponding vertex $P(w_3)$ in its interior.
	\item We add a  vertex $w_{12}$ on the edge $[z_0,z_3]$.   The  edge $[z_0,z_3]$ is divided into two edges $[z_0,w_{12}]$ and $[z_3,w_{12}]$.  Note that the map $R$ interchanges the edges $[z_0,w_{12}]$  and $[w_{12},z_3]$.

\end{enumerate}
See Figure 	\ref{figure:fundamentaldomain} for these vertices. 

We introduce more edges in the interiors of some  ridges of $D$ for the new  polytope $D^*$.
\begin{enumerate}
	\item We add an edge $[z_2,w_{12}]$ in the ridge  $[z_0,z_2,z_3]$, which is a geodesic arc connecting 
	$z_2 $ and $w_{12}$.
	
	Since the ridge $[z_0,z_2,z_3]$ is paired to the ridge $[z_0,z_1,z_3]$ by $R^{-1}_1$, there is an edge $[z_1,w_{12}]$ in the ridge  $[z_0,z_1,z_3]$, which is equivalent to $[z_2,w_{12}]$ by $R^{-1}_1$.

	\item   We add an edge $[w_3,w_4]$ in the ridge  $[z_0,z_1,z_2]$, which is a geodesic arc connecting 
	$w_3$ and $w_{4}$.
	
	 Recall that $J(w_3)=w_3$,  $J(z_0)=z_1$,  $J(z_1)=z_2$ and  $J(z_2)=z_0$. There are also two edges in  the ridge  $[z_0,z_1,z_2]$, which are  images of the edge $[w_3,w_4]$ by $J$ and $J^2$. We denote them by $[w_3,J(w_4)]=J([w_3,w_4])$ and   $[w_3,PJ(w_4)]=J^2([w_3,w_4])$.

\end{enumerate}

We also introduce   new ridges for the new  polytope $D^*$.
\begin{enumerate}
\item  The triangular ridge $[z_0,z_1,z_2]$ in $D$ is divided into three 
       quadrilaterals in the polytope $D^*$, which are cyclically permuted by the map $J$. One of  them has vertices $z_2$, $w_4$, $w_3$ and $PJ(w_4)$, so we denote this quadrilateral by  $[z_2,w_4,w_3, PJ(w_4)]$.
       
        Similarly,  the ridge $[z_1,z_2,z_3]$ in $D$ is divided into three 
        quadrilaterals in the polytope $D^*$.  Note that the six
       quadrilaterals obtained  are equivalent under  the action of  $\rm PU(2,1; \mathbb{Z}[\omega])$.

 \item  The triangular ridge $[z_0,z_2,z_3]$ in $D$ is divided into two
triangles in the polytope $D^*$ by the edge $[z_2,w_{12}]$.  One of  them has vertices $z_0$, $z_2$, and $w_{12}$, so we denote this triangle by  $[z_0, z_2, w_{12}]$.  The map $R$ interchanges this two triangles.

 Similarly,  the ridge $[z_0,z_1,z_3]$ in $D$ is divided into two triangles
 in the polytope $D^*$. Note that the four sub-triangles of $[z_0,z_1,z_3]$ and $[z_0,z_2,z_3]$ are equivalent under  the action of  $\rm PU(2,1; \mathbb{Z}[\omega])$.
 
 \item We add a new ridge $[z_2, z_1,w_{12}]$  in the new  polytope $D^*$.  By Proposition  \ref{prop:reflection}, each of $z_2$,  $z_1$ and $w_{12}$ is fixed by the $\mathbb{C}$-reflection $R$, so we let $[z_2,z_1,w_{12}]$  be the geodesic triangle with vertices  $z_2$ $z_1$ and  $w_{12}$ in the mirror of  $R$.  $[z_2,z_1,w_{12}]$ is the green triangle in 	Figure \ref{figure:fundamentaldomain} with vertices $w_{12}$, $z_1$ and $z_2$.

\end{enumerate}

At last we introduce more  3-cells for the new  polytope $D^*$.

The 3-cell $[z_0,z_1,z_2,z_3]$ in $D$ is divided into two parts by the ridge $[z_1, z_2,w_{12}]$. One of  them has vertices $z_0$, $z_1$, $z_2$ and $w_{12}$ (and $w_3$, $w_4$, $J(w_4)$, $PJ(w_4)$), so we denote this 3-cell by  $$[w_{12}, z_0, z_1, z_2, (w_3, w_4, J(w_4), PJ(w_4))].$$  The map $R$ interchanges these two 3-cells.

	Now we denote by $D^{*}$ the polytope which is obtained from $D$ by subdivide some of its $k$-face for $k \leq 3$:
	
\begin{enumerate}
	\item the underlying topological spaces of $D$ and $D^{*}$ are the same. 
	
		\item There are  four infinite  3-faces  and two finite 3-faces of $D^{*}$.
		
		 For example we denote one of the infinite  3-faces of $D^{*}$ by $[q_{\infty}, z_0,z_1,z_2,(w_3, w_4, J(w_4),PJ(w_4))]$. 
		The underlying topological space of $[q_{\infty}, z_0,z_1,z_2,(w_3, w_4, J(w_4),PJ(w_4))]$ is the  underlying topological space of the  3-face $[q_{\infty}, z_0,z_1,z_2]$ of $D$, but we subdivide the  2-face $[z_0,z_1,z_2]$ into three   quadrilaterals. Moreover, due to the added new vertices, the remaining three  2-faces are now viewed as  three   quadrilaterals.
		 See the left sub-figure of Figure  \ref{figure:3faceDstar},  which is
		 a  3-face of $D^{*}$.
		 
		 Similarly,   $[q_{\infty}, z_1,z_2,z_3,(P(w_3), P(w_4), PJ(w_4),P^2J(w_4))]$ is an  infinite  3-face of $D^{*}$. 
		 There is also a 3-face denoted  by $$[q_{\infty}, z_0,z_2,z_3,(w_4, P^2J(w_4), w_{12})].$$
		 The underlying topological space of  $$[q_{\infty}, z_0,z_2,z_3,(w_4, P^2J(w_4), w_{12})]$$ is the  underlying topological space of the  3-face $[q_{\infty}, z_0,z_2,z_3)]$ of $D$, but we subdivide the  2-face $[z_0,z_2,z_3]$ into three   triangles. Moreover, due to the added new vertices, the remaining three  2-faces are now viewed as  three   quadrilaterals.
		 See the right sub-figure of Figure  \ref{figure:3faceDstar}, which 
		 is a  3-face of $D^{*}$.
		 
		 Similarly,   $[q_{\infty}, z_0,z_1,z_3,(w_4, J(w_4), w_{12})]$ is an  infinite  3-face of $D^{*}$. 
		 One of the finite   3-faces of $D^{*}$ is denoted  by $$[w_{12}, z_0,z_1,z_2,(w_3, w_4, J(w_4),PJ(w_4))].$$
		 The  underlying topological space of it  is the  underlying topological space of the  3-face $[w_{12}, z_0,z_1,z_2]$, which is exactly half of the 3-face  $[z_3, z_0, z_1, z_2]$ of  $D$, but we subdivide the  2-face $[z_0,z_1,z_2]$ into three   quadrilaterals. Moreover, due to the added new vertices, the remaining three  2-faces are now viewed as  three   quadrilaterals. The combinatoric of $$[w_{12}, z_0,z_1,z_2,(w_3, w_4, J(w_4),PJ(w_4))]$$
		is the same as the combinatoric of $$[q_{\infty}, z_0,z_1,z_2,(w_3, w_4, J(w_4),PJ(w_4))],$$ we only need to change $q_{\infty}$ to $w_{12}$.
		 
		 Similarly,   $[w_{12}, z_1,z_2,z_3,(P(w_3), P(w_4), PJ(w_4),P^2J(w_4))]$ is  a finite  3-face of $D^{*}$.

			\item   $k$-faces of $D^{*}$ for $0 \leq k \leq 2$ are now obvious, we omit the details.	
\end{enumerate}


\begin{figure}
	\begin{center}
		\begin{tikzpicture}
		\node at (0,0) {\includegraphics[width=11cm,height=4.5cm]{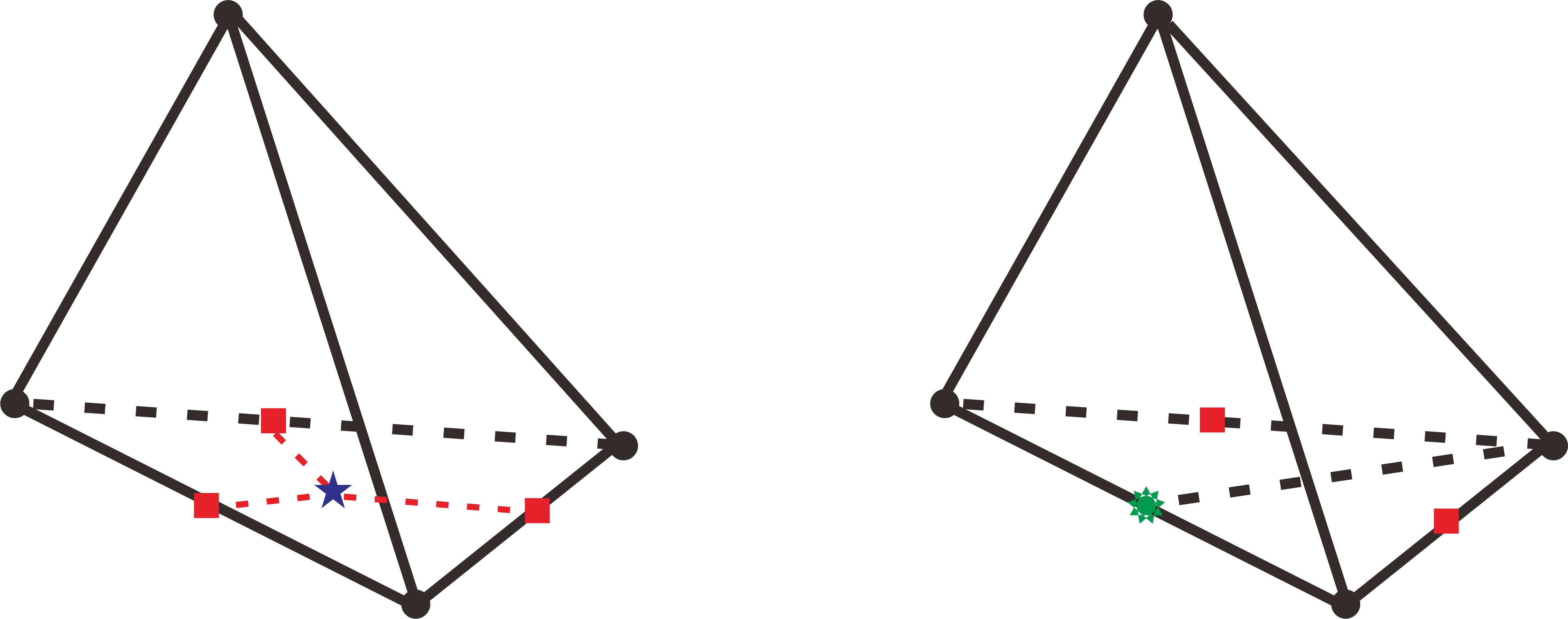}};
		
\node at (-0.6,-0.5){\large $z_0$};
		\node at (-2.3,-2.5){\large $z_1$};
		\node at (-5.4,-1.3){\large $z_2$};
		\node at (-3.2,2.3){\large $q_{\infty}$};
		\node at (-3.6,-0.4){\tiny $w_4$};
				\node at (-1.5,-1.9){\tiny $J(w_4)$};
		\node at (-4.0,-1.9){\tiny $PJ(w_4)$};
			\node at (-3.6,-1.2){\tiny $w_3$};
			
			\node at (0.9,-0.3){\large $z_0$};
			\node at (5.8,-1.3){\large $z_2$};
			\node at (4.4,-2.5){\large $z_3$};
			\node at (1.7,2.3){\large $q_{\infty}$};
			\node at (2.7,-0.4){\tiny $w_4$};
			\node at (5.1,-1.9){\tiny $P^2J(w_4)$};
			\node at (2.5,-1.8){\tiny $w_{12}$};
		
		\end{tikzpicture}
	\end{center}
\caption{Two infinite 3-faces  of the fundamental domain $D^*$.}
	\label{figure:3faceDstar}
\end{figure}
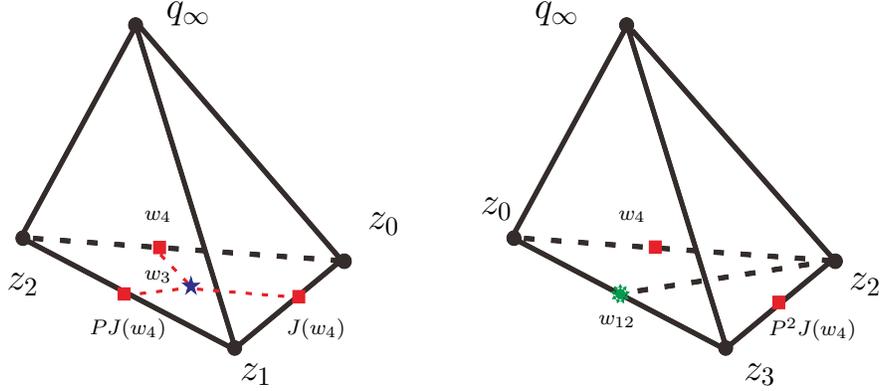

\subsection{Side-pairings on the new  polytope $D^*$.}

From the combinatoric of $D^{*}$ above and the gluing-pattern of the fundamental domain $D$, we  have the side-parings of $D^{*}$ as follows:




\begin{enumerate}

\item 	\begin{gather}
\nonumber  P: [q_{\infty}, z_0, z_1, z_2, (w_3, w_4, J(w_4), PJ(w_4))]\longrightarrow  \\
\nonumber  [q_{\infty}, z_1, z_2, z_3, (P(w_3), P(w_4), PJ(w_4), P^2J(w_4))];
	\end{gather}
	\item  		\begin{gather}
	\nonumber  PQ^{-1}: [q_{\infty}, z_0, z_2, z_3, (w_4, P^2J(w_4), w_{12})]\longrightarrow 
	  \\
	\nonumber  [q_{\infty}, z_0, z_1, z_3, (J(w_4), P(w_4), w_{12})];
	\end{gather}

	\item		\begin{gather}
	\nonumber R: [w_{12}, z_0, z_1, z_2, (w_3, w_4, J(w_4), PJ(w_4))]\longrightarrow 
	  \\
	 \nonumber [w_{12}, z_3, z_1, z_2, (P(w_3), P^2J(w_4),P(w_4), PJ(w_4))].
	\end{gather}

\end{enumerate}

We then  consider tessellation about ridges.
We use  $[q_{\infty},z_2,z_0,(w_4)]$ to denote the   quadrilateral in $D^{*}$ which has the same underlying topological space as  $[q_{\infty},z_2,z_0]$ in $D$, but with a new vertex $w_4$ added to it.

(1).  The   ridge cycle involved $[q_{\infty},z_2,z_0,(w_4)]$ is
\begin{flalign}
\nonumber [q_{\infty},z_2,z_0,(w_4)]  \xrightarrow{P} [q_{\infty},z_3,z_1,(P(w_4))] \xrightarrow{QP^{-1}}  \\ 
\nonumber [q_{\infty},z_3,z_2,(P^2J(w_4))]  \xrightarrow{P^{-1}}  [q_{\infty},z_2,z_1,(PJ(w_4))]  \\ 
\nonumber  \xrightarrow{P^{-1}}  [q_{\infty},z_1,z_0,(J(w_4))]\xrightarrow{QP^{-1}}  [q_{\infty},z_2,z_0,(w_4)];
\end{flalign}

Similarly, 

(2). The second  ridge cycle is
\begin{flalign}
\nonumber [z_2, w_4, w_3, (PJ(w_4))]  \xrightarrow{P} [z_3, P(w_4),  P(w_3), (P^2J(w_4))] \xrightarrow{R^{-1}}\\\nonumber [z_0, J(w_4), w_3, (w_4)] \xrightarrow{P}  [z_1, PJ(w_4), P(w_3), (P(w_4))]\xrightarrow{R^{-1}} \\\nonumber   [z_1, PJ(w_4), w_3, (J(w_4))] \xrightarrow{P}  [z_2, P^2J(w_4), P(w_3), (PJ(w_4))] \\\nonumber \xrightarrow{R^{-1}} [z_2, w_4, w_3, (PJ(w_4))];
\end{flalign}

(3). The third   ridge cycle is
\begin{flalign}
\nonumber [z_0, w_{12}, z_2,(w_4)]  \xrightarrow{PQ^{-1}} [z_0, w_{12}, z_1, (J(w_4))] \xrightarrow{R} \\
 \nonumber   [z_3, w_{12}, z_1, (P(w_4))]   \xrightarrow{QP^{-1}}  [z_3, w_{12}, z_2 ,(P^2Jw_4)]  \xrightarrow{R^{-1}}  [z_0, w_{12}, z_2, (w_4)];
\end{flalign}

(4). The fourth  ridge cycle is
\begin{flalign}
\nonumber [z_1,z_2,w_{12}]  \xrightarrow{R} [z_1,z_2, w_{12}];
\end{flalign}

(5). The fifth ridge cycle is
\begin{flalign}
\nonumber [q_{\infty},z_0,z_3, (w_{12})]  \xrightarrow{PQ^{-1}} [q_{\infty}, z_0, z_3, (w_{12})].
\end{flalign}

From Subsection \ref{subsection:equaclassD} and the ridge circles of $D^{*}$, we can get Table  \ref{table:equaclassDstar}, which list the equivalence classes of  $k$-dimension faces of $D^{*}$ under the 
 action of  $\Gamma=\rm PU(2,1; \mathbb{Z}[\omega])$.

	

\begin{table}[h]
		\caption{The equivalence classes of  $k$-dimension faces of $D^{*}$.}
	\label{table:equaclassDstar}
	\centering
	\begin{tabular}{l|l|c}
		\hline
		
		\textbf{Dimension} &  \textbf{Number}      & \textbf{Equivalence class}    \\
		
			\hline
		\multirow{2}*{0} & \multirow{2}*{5}  & $[q_{\infty}]$ \\
		\cline{3-3}
		&        & $[z_0]$, $[w_{3}]$, $[w_{4}]$, $[w_{12}]$  \\
		\hline
	\multirow{2}*{1} & \multirow{2}*{5}  &   $[z_0, q_{\infty}]$ \\
	\cline{3-3}
	&        & $[z_0,w_{12}]$, $[z_1,w_{12}]$, $[z_0,w_4]$,   $[w_3,w_4]$  \\
	\hline
	
	\multirow{2}*{2} & \multirow{2}*{5}  & $[z_0,w_4, z_2, q_{\infty}]$, $[z_0,w_{12}, z_3,  q_{\infty}]$\\
	\cline{3-3}
	&        & $[z_0,w_4,w_3,J(w_4)]$,  $[z_0,J(w_4), z_1,w_{12}]$,   \\
	&        &    $[z_1,w_{12},z_2, PJ(w_4)]$ \\
	\hline
	
	\multirow{2}*{3} & \multirow{2}*{3}  &  $[q_{\infty}, z_0, z_1, z_2, (w_3, w_4, J(w_4), PJ(w_4))]$, \\  &        & $[q_{\infty}, z_0, z_2, z_3, (w_4, P^2J(w_4), w_{12})]$   \\
	\cline{3-3}
	&        &  	$[w_{12}, z_0, z_1, z_2, (w_3, w_4, J(w_4), PJ(w_4))]$  \\
	\hline
{4} & {1}  &  $[z_0,z_1,z_2,z_3,q_{\infty}]$ \\
	
	\hline
	
	\end{tabular}

\end{table}

We note that in Table 	\ref{table:equaclassDstar}, the edges $[z_0,w_{12}]$ and $[z_1,w_{12}]$ are in different classes. $[z_0,w_{12}]$ is a sub-edge of the edge $[z_0,z_{3}]$ of $D$, but  $[z_1,w_{12}]$ lies in the interior of a ridge of $D$.

	\subsection{A handle structure of the Eisenstein-Picard modular surface $M$}\label{subsection:handleEPmodular}

	We now consider the global topology of the Eisenstein-Picard modular surface $M$.

We take a horosphere $H_{u}$ in  ${\bf H}^{2}_{\mathbb C}$ with  $u$ large enough,  $H_{> u}$ is the associated horoball. Then $\Gamma$ acts on  
$${\bf H}^{2}_{\mathbb C}- \mathop{\cup}_{\gamma \in \Gamma} \gamma \cdot  H_{ > u}$$
  proper discontinuously, the quotient space is denoted by $M_{\leq u}$. We can  view $M_{\leq u}$ as the space as truncating the cusp of $M$. So $M_{\leq u}$ is a compact 4-orbifold with non-empty boundary,  and $M$ is homeomorphic to the interior of $M_{\leq u}$. Then  $M_{\leq u}$ is a compact core of $M$.

On the other hand, $M$ is the quotient space of  $$[z_0,z_1,z_2,z_3, q_{\infty}]-q_{\infty}$$ by side-parings.  Topologically,  $[z_0,z_1,z_2,z_3, q_{\infty}]-q_{\infty}$ is homeomophic to  $[z_0,z_1,z_2,z_3] \times [0,\infty)$, we fix such a homeomorphism.   In particular,  $[z_0,z_1,z_2,z_3, q_{\infty}]-q_{\infty}$   is homeomophic to its subset  $[z_0,z_1,z_2,z_3] \times [0,t)$ for any $t >0$. Moreover, the homeomorphism  is $\Gamma$-invariant.
In particular, if we consider side-parings on $[z_0,z_1,z_2,z_3] \times [0,t]$,  which are restrictions of side-parings on $[z_0,z_1,z_2,z_3, q_{\infty}]-q_{\infty}$. So  $M_{\leq u}$  is homeomophic to the quotient space of  $[z_0,z_1,z_2,z_3] \times [0,t]$  by side-parings for any $t>0$.

 Consider the tetrahedron  $[z_0,z_1,z_2,z_3]$ in ${\bf H}^{2}_{\mathbb C}$, the quotient space of it is a 3-dimensional CW-complex $Y$ in $M$. We should note some 2-cells of $Y$ are totally geodesic in $M$, which are crucial for the topology of $M$. 
We take a  closed  neighborhood $N$ of   $Y$ in $M$, which is the quotient space of $[z_0,z_1,z_2,z_3]\times [0,\epsilon]$ for $\epsilon >0$ small enough. So  $M_{\leq u}$ is homeomorphic to  $N$, and then   $N$ is a compact core of $M$.

We now consider the topology of $N$.  We first sketch the process, then we will give more details later. Since  $N$ can be viewed as the union  of small   closed  neighborhoods $N(c)$ in $M$ for each cell $c$  in Table 	\ref{table:equaclassDstar} without $q_{\infty}$ as  one of  its vertices. That is, we may take $c$ be any of  $$[z_0], ~~[w_3],~~[w_4],~~ [w_{12}],$$ 
$$[z_0,w_{12}],~~[z_1,w_{12}],~~[z_0,w_4],~~[w_3,w_4],$$  $$[z_0,w_4,w_3,J(w_4)],~~[z_0,J(w_4), z_1,w_{12}],~~[z_1,w_{12},z_2, w_4]$$ and 	$$[w_{12}, z_0, z_1, z_2, (w_3, w_4, J(w_4), PJ(w_4))].$$
By this we mean we first take  small   closed  neighborhoods 
$$N([z_0]), ~~N([w_3]),~~N([w_4]),~~N([w_{12}]),$$  in $M$. We will view them as orbifold 0-handles.

 Then we take  small   closed  neighborhoods $$N([z_0,w_{12}]),~~N([z_1,w_{12}]),~~N([z_0,w_4]),~~N([w_3,w_4])$$
 of edges  $[z_0,w_{12}],~~[z_1,w_{12}],~~[z_0,w_4],~~[w_3,w_4]$ in $M$.
 But for technical reasons,  what we really need are $$N([z_0,w_{12}])-int(N([z_{0}])\cup N([w_{12}])),$$
$$N([z_1,w_{12}])- int(N([z_1]) \cup N([w_{12}])),$$  $$N([z_0,w_4])-int(N([z_0])\cup N([w_4])),$$ and $$N([w_3,w_4])-int(N([w_3])\cup N([w_4]))$$ respectively. We will denote them by $$h([z_0,w_{12}]), ~~h([z_1,w_{12}]),~~h([z_0,w_{4}]),~~ h([w_3,w_4]).$$ We will view them as orbifold 1-handles. But for the simplicity and transparency of the topology of $M$,  we will  assimilate the handles $$h([z_1,w_{12}]),~~h([z_0,w_{4}]),~~ h([w_3,w_4])$$ into the product of a pair of  pants and a 2-orbifold  $\mathcal{F}_2$. 

Similarly,   we take a  small   closed  neighborhood $$N([z_1,w_{12},z_2, PJ(w_4)])$$ 
of $[z_1,w_{12},z_2, PJ(w_4)]$ in $M$.
What we really need is the complement of  $$int(N([z_{0}])\cup N([w_{4}]\cup  N([w_{12}]) \cup h([z_1,w_{12}]) \cup h([z_0,w_{4}]) \cup  h([w_3,w_4]))$$ in
$N([z_1,w_{12},z_2, PJ(w_4)])$. 
It will give us  the product of a    pair of pants and a 2-orbifold  $\mathcal{F}_2$.

\textbf{Step 1.} From Section \ref{sec:localmodel},  each of $w_3$, $w_4$, $w_{12}$ and $z_0$ is the fixed point of an isotropy group of order $3$, $4$, $12$ and $72$ respectively.
So  the closed neighborhoods of $[z_0]$,  $[w_3]$, $[w_4]$ and   $[w_{12}]$ in $M$ are 4-orbifolds with finite orbifold-fundamental groups. We  view each of the  closed neighborhoods as  orbifold 0-handles $h(w_3)$, $h(w_4)$,  $h(w_{12})$ and $h(z_0)$ respectively.
	 Recall      the topologies of these orbifold 0-handles in Subsections  \ref{subsection:w3}, \ref{subsection:w4}, \ref{subsection:w12} and  \ref{subsection:z0}.

\textbf{Step 2.}  The  edge  $[w_3,w_4]$ in $Y$ is easy for our process, since there is no singular set in the interior of it.  We take a sub-interval of $[w_3,w_4]$, say  $[w_3,w_4]- int(h(w_3) \cup h(w_4))$.  The complement of  $h(w_3) \cup h(w_4)$ in a closed neighborhood of $[w_3,w_4]$  in $M$  is  a 1-handle  $h([w_3,w_4])$. Where  $[w_3,w_4]- int(h(w_3) \cup h(w_4))$  corresponds to $\mathbb{D}^1$-factor of  the 1-handle $\mathbb{D}^1 \times \mathbb{D}^3$ attached to  orbifold 0-handles $h(w_3)$ and $ h(w_4)$.  
Moreover, we should remark that 	attaching sphere of   $h([w_3,w_4])$ is disjoint from the singular set of  $h(w_3) \cup h(w_4)$.

We then consider the    orbifold 1-handle  $h([z_0,w_{12}])$, which is a little complicated.

We denote by $\mathbb{C}_6$  the $\mathbb{C}$-line fixed by $PQ^{-1}$.  The intersection of $\mathbb{C}_6$  and the fundamental domain $D$  is the cyan  triangle  $[z_0,z_3, q_{\infty}]$ in Figure 	\ref{figure:fundamentaldomain} with vertices $z_0$, $z_3$ and $q_{\infty}$.
 Note that the vertex $w_{12}$ lies in the interior of  the edge $[z_0,z_3]$, and  the involution $R$ identifies  $[z_0,w_{12}]$ and $[z_3, w_{12}]$. So the quotient space of  $$[z_0,z_3, q_{\infty}]-(z_0 \cup z_3 \cup w_{12}\cup  q_{\infty})$$ in $M$ is a pair of  totally geodesic pants $P_6$, which is a $\mathbb{Z}_6$-coned singular set. See Figure 	\ref{figure:hz0w12}, we first delete the four vertices  $z_0, z_3,  w_{12}$ and $q_{\infty}$ of the triangle in left sub-figure of Figure 	\ref{figure:hz0w12}, then  identifying the two open edges $(z_0,w_{12})$ and  $(z_3,w_{12})$, we get a topological pants.

   Consider the red path and the green path in   Figure 	\ref{figure:hz0w12}, which give two isotopic curves in $P_6$.  Moreover, the red path in   Figure 	\ref{figure:hz0w12} is the intersection $H_{u}\cap [z_0,z_3, q_{\infty}]$, and the green  path in   Figure 	\ref{figure:hz0w12} is the intersection
  $N\cap [z_0,z_3, q_{\infty}]$. 
 Then we consider  neighborhoods $N_{M_{\leq u}}(P_6)$ and  $N_{N}(P_6)$ of $P_6$ in $M_{\leq u}$ and  $N$ respectively. So $N_{M_{\leq u}}(P_6)\cap P_6 $ and $N_{N}(P_6)\cap P_6$ are homeomorphic by homeomorphism between $M_{\leq u}$ and $N$. 
   In other words, consider the homeomorphism between $M_{\leq u}$ and $N$,  we  delete the region co-bound by the red and green curves in $P_6$. 
    So what we shall do is gluing the two disks in the left sub-figure of  Figure 	\ref{figure:hz0w12} co-bounded by the green and cyan arcs (and orange, purple arcs). But in fact, these two disks are glued together into a bigger disk, say $E$,  as  in the right sub-figure of  Figure 	\ref{figure:hz0w12}.
   
    The complement of  $int(h(w_3) \cup h(w_4) \cup h(w_{12}))$ in a closed neighborhood of $P_6$  in $M$  is  an orbifold  1-handle  $h([w_{12},z_0])$. That is, gluing a copy of  $\mathbb{D}_1^1 \times \mathbb{D}^1_2 \times \mathcal{F}_6$ to the 4-orbifold we have got. Where the disk $E$ corresponds to $\mathbb{D}_1^1 \times \mathbb{D}^1_2 \times \{c_6\}$, and  $c_6$ is the coned point of  $\mathcal{F}_6$. Moreover, the two orange and purple arcs indicate $\partial \mathbb{D}_1^1 \times \mathbb{D}^1_2$.

	We remark that 	attaching the orbifold 1-handle  $h([z_0,w_{12}])$  needing more carefully, by this we mean that the $\mathbb{Z}_6$-coned singular set of  $h([z_0,w_{12}])$ must be glued to  $\mathbb{Z}_2$-coned singular sets of $h(w_4)$ and $h(w_{12})$.

	\begin{figure}
		\begin{center}
			\begin{tikzpicture}
			\node at (0,0) {\includegraphics[width=11cm,height=4cm]{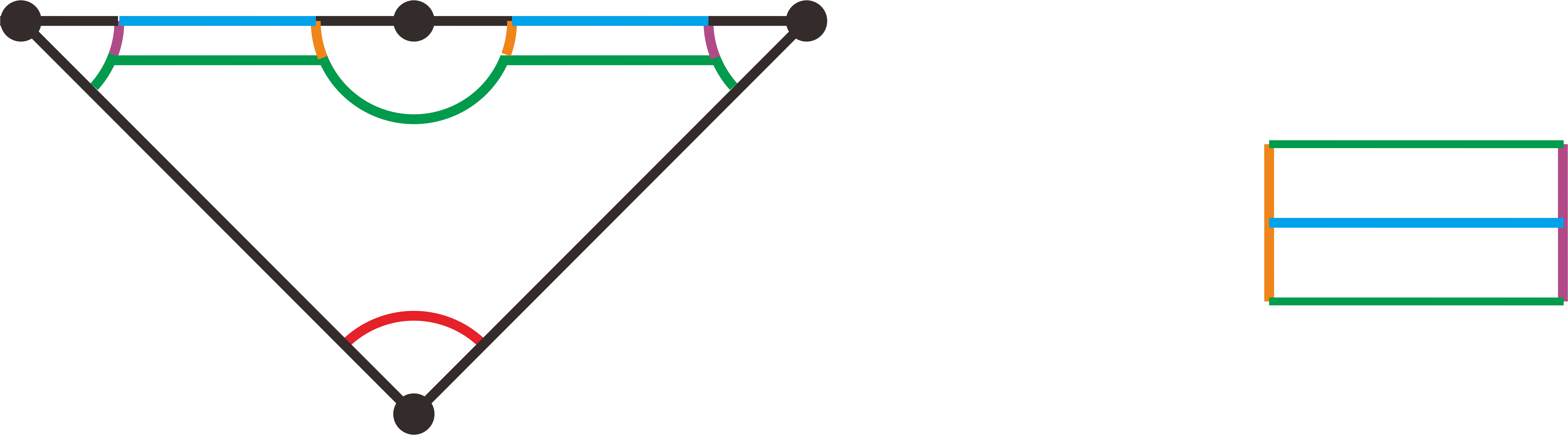}};

			\node at (-5.2,2.3){\large $z_0$};
			\node at (0.5,2.3){\large $z_3$};
			\node at (-1.89,2.3){\large $w_{12}$};
			\node at (-1.9,-1.9){\large $q_{\infty}$};
			\end{tikzpicture}
		\end{center}
		\caption{The totally geodesic triangle $[z_0,z_3,q_{\infty}]$ gives the orbifold 1-handle $h([z_0,w_{12}])$.}
		\label{figure:hz0w12}
	\end{figure}
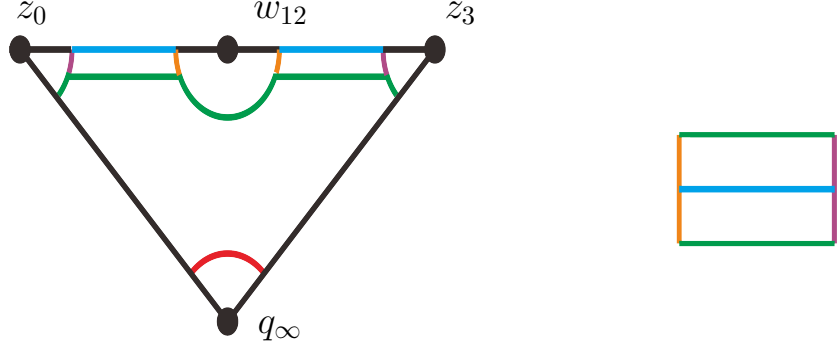

\textbf{Step 3.}  Attach the product of a  pair of   pants and  $\mathcal{F}_2$, namely $h([z_0,w_4,w_{12}])$, 
	to get a 4-orbifold with 0-dimensional and  2-dimensional singular sets.

We denote by $\mathbb{C}_2$  the $\mathbb{C}$-line fixed by $R$.  The intersection of $\mathbb{C}_2$  and the fundamental domain $D$  is the green  triangle $[z_0,z_2,w_{12}]$ in Figure 	\ref{figure:fundamentaldomain} with vertices $z_1$, $z_2$ and $w_{12}$.   Note that the vertex $PJ(w_4)$ lies in the interior of  the edge $[z_1,z_2]$. Moreover, $PQ^{-1}$  identifies  $[z_1,w_{12}]$ and $[z_2, w_{12}]$. 

	So the quotient space of  $$[z_1,z_2, w_{12}]-(z_1 \cup z_2 \cup w_{12}\cup  PJ(w_{4})$$ in $M$ is a pair of  totally geodesic pants $P_2$, which is a $\mathbb{Z}_2$-coned singular set. Which is similar to  the pair of pants $P_6$.

	The complement of  $int(h(w_4) \cup h(w_{12}) \cup h(z_{0}))$ in a closed neighborhood of $p_6$  in $M$  is    the product of a pair of pants and  $\mathcal{F}_2$, we denote it by $h([z_0,w_4,w_{12}])$. Where $P_2-(h(w_4) \cup h(w_{12}) \cup h(z_{0}))$ is the pair of  pants here.

We remark that   $h([z_0,w_4,w_{12}])$  is a $\mathbb{Z}_2$-coned orbifold. So the singular set of  it  must be glued to  $\mathbb{Z}_2$-coned singular sets of $h(w_4)$, $h(w_{12})$ and  $h(z_{0})$. 
	
In other words, we  view the union of closed neighborhoods of   $[z_0,w_4]$,  $[z_1,w_{12}]$ and  $[z_0,w_4,w_{12}]$  as a whole object $h([z_0,w_4,w_{12}])$. So  $[z_0,w_4]$ and  $[z_1,w_{12}]$ do not contribute new orbifold 1-handles anymore.

\textbf{Step 4.}  Attach two  2-handles  $$h([z_0,w_4,w_3,J(w_4)]) \quad \text{and} \quad h([z_0,J(w_4), z_1,w_{12}]).$$

The complement of  $int(h(w_4) \cup h(w_{12}) \cup h(z_{0}))$ in a closed neighborhood of  $[z_0,w_4,w_3,J(w_4)]$  in $N$  is    the product of two 2-disks. So it gives a 2-handle in $N$. Where $$[z_0,w_4,w_3,J(w_4)]-int(h(w_4) \cup h(w_{12}) \cup h(z_{0}))$$ is the disk $\mathbb{D}^2_1$ in the 2-handle $\mathbb{D}^2_1 \times \mathbb{D}^2_2$. 

 Similarly,  the complement of  $int(h(w_4) \cup h(w_{12}) \cup h(z_{0}))$ in a closed neighborhood of  $[z_0,J(w_4), z_1,w_{12}]$  in $N$  is    the product of two 2-disks, and  it gives a 2-handle in $N$. Where $$[z_0,J(w_4), z_1,w_{12}]-int(h(w_4) \cup h(w_{12}) \cup h(z_{0}))$$ is the disk $\mathbb{D}^2_1$ in the 2-handle $\mathbb{D}^2_1 \times \mathbb{D}^2_2$.

We remark that   attaching 2-handles  $$h([z_0,J(w_4), z_1,w_{12}) \quad  \text{and} \quad h([z_0,w_4,w_3,w_{12}])$$ must be along two well-chosen  curves on the 4-orbifold obtained after Step 3. These curves are hidden in the side-pairing pattern of the  fundamental domain $D^*$ of  ${\bf H}^2_{\mathbb C}/\rm PU(2,1; \mathcal{O}_3)$.
	
\textbf{Step 5.}  At last we  attach a 3-handle  $$h([w_{12}, z_0, z_1, z_2, (w_3, w_4, J(w_4), PJ(w_4))]).$$

The complement of  $int(h(w_3) \cup h(w_4) \cup h(w_{12}) \cup h(z_{0}))$ in a closed neighborhood of  $[w_{12}, z_0, z_1, z_2, (w_3, w_4, J(w_4), PJ(w_4))]$  in $N$  is    the product of $\mathbb{D}^3 \times \mathbb{D}^1$. So it gives a 3-handle in $N$. Where $$[w_{12}, z_0, z_1, z_2, (w_3, w_4, J(w_4), PJ(w_4))]-(h(w_3) \cup h(w_4) \cup h(w_{12}) \cup h(z_{0}))$$ is the disk $\mathbb{D}^3$ in the 3-handle.

The reader may prefer to Figure \ref{figure:EPmodular} for the procedure to get the topology of $M$ in Theorem \ref{thm:modulartop}.
This ends the proof of  Theorem \ref{thm:modulartop}.



\end{document}